\documentclass[review,onefignum,onetabnum]{siamart250211}

\usepackage{lipsum}
\usepackage{amsfonts}
\usepackage{graphicx}
\usepackage{epstopdf}
\usepackage{algorithmic}
\usepackage{amssymb}
\usepackage[english]{babel}
\usepackage{array}
\usepackage{adjustbox}
\usepackage{relsize}
\usepackage{spverbatim}
\usepackage{listings}
\usepackage{mathrsfs}
\usepackage{derivative}
\usepackage{mathtools}
\usepackage{float}
\usepackage[section]{placeins}

\usepackage{hyperref}
\usepackage{cleveref}

\usepackage{etoolbox}

\makeatletter
\newcommand{\crefcomma}[1]{%
  \begingroup
    \def\crefcomma@sep{}%
    \forcsvlist{\crefcomma@do}{#1}%
  \endgroup
}
\newcommand{\crefcomma@do}[1]{%
  \ifx\crefcomma@sep\@empty\else,~\fi
  \cref{#1}%
  \def\crefcomma@sep{,}%
}
\newcommand{\Crefcomma}[1]{%
  \begingroup
    \def\crefcomma@sep{}%
    \forcsvlist{\Crefcomma@do}{#1}%
  \endgroup
}
\newcommand{\Crefcomma@do}[1]{%
  \ifx\crefcomma@sep\@empty\else,~\fi
  \Cref{#1}%
  \def\crefcomma@sep{,}%
}
\makeatother

\DeclareMathAlphabet{\mathpzc}{OT1}{pzc}{m}{it}

\newcolumntype{L}{>{$}l<{$}}

\ifpdf
  \DeclareGraphicsExtensions{.eps,.pdf,.png,.jpg}
\else
  \DeclareGraphicsExtensions{.eps}
\fi

\newsiamremark{hypothesis}{Hypothesis}
\crefname{hypothesis}{Hypothesis}{Hypotheses}
\newsiamthm{claim}{Claim}
\newtheorem{remark}{Remark}
\nolinenumbers

\headers{The general Brannan coefficient conjecture II}{T. M. Dunster}

\title{The general Brannan coefficient conjecture II: Meijer-function approximations}

\author{T. M. Dunster\thanks{Department of Mathematics and Statistics, San Diego State University, 5500 Campanile Drive, San Diego, CA 92182-7720, USA. 
  (\email{mdunster@sdsu.edu}, \url{https://tmdunster.sdsu.edu}).}
  }

\usepackage{amsopn}

\makeatletter
\newcommand*{\addFileDependency}[1]{% argument=file name and extension
  \typeout{(#1)}% latexmk will find this if $recorder=0 (however, in that case, it will ignore #1 if it is a .aux or .pdf file etc and it exists! if it doesn't exist, it will appear in the list of dependents regardless)
  \@addtofilelist{#1}% if you want it to appear in \listfiles, not really necessary and latexmk doesn't use this
  \IfFileExists{#1}{}{\typeout{No file #1.}}% latexmk will find this message if #1 doesn't exist (yet)
}
\makeatother

\nolinenumbers

\ifpdf
\hypersetup{
  pdftitle={The general Brannan coefficient conjecture II: Meijer-function approximations},
  pdfauthor={T. M. Dunster}
}
\fi

\begin{document}

\maketitle

\begin{abstract}
The coefficients $A_n(\alpha,\beta,\omega)$ in the Maclaurin expansion $(1+\omega z)^{\alpha}(1-z)^{-\beta}=\sum_{n=0}^{\infty} A_n(\alpha,\beta,\omega)z^n$ are considered for $|\omega|=1$ and $\alpha,\beta\in(0,1]$. D. A. Brannan conjectured in a 1973 paper that $|A_n(\alpha,\beta,\omega)|\le A_n(\alpha,\beta,1)$ for every positive odd integer $n$. The present author recently established the conjecture outside a small neighbourhood of $\omega=-1$. The remaining range is treated here by combining compound Laplace integral representations with two types of local approximation: a Meijer $G$-function approximation for $n|\arg(-\omega)|$ bounded, and a modified Watson approximation for the complementary range. The resulting lower bounds reduce the problem to numerical positivity checks for explicit functions on compact parameter sets. These computations verify the inequality for all $\alpha,\beta\in(0,1]$ and all odd integers $n\ge5$, and hence, together with Brannan's result for $n=3$, complete the proof of his conjecture.
\end{abstract}



\section{Introduction}

We begin by recalling Brannan's coefficient conjecture and reducing it to an equivalent inequality for a hypergeometric function. For $n=1,2,3,\ldots$, define $A_n(\alpha,\beta,\omega)$ by
%%%%%%%%%%%%%%%%%
\begin{equation}
\label{eq01}
\frac{(1+\omega z)^{\alpha}}{(1-z)^{\beta}}
= \sum_{n=0}^{\infty} A_n(\alpha,\beta,\omega)z^n,
\end{equation}
%%%%%%%%%%%%%%%%%
where $z,\omega\in\mathbb{C}$, $|z|<1=|\omega|$, and $\alpha,\beta\in(0,1]$.

The purpose of this paper is to prove the following.

\begin{theorem}
\label{thm:Brannan}
Suppose $\alpha,\beta\in(0,1]$ and $0\le |\theta|\le \pi$. Then
%%%%%%%%%%%%%%%%%
\begin{equation}
\label{eq02}
\left|A_n(\alpha,\beta,e^{i\theta})\right|
\le A_n(\alpha,\beta,1)
\quad (n=3,5,7,\ldots).
\end{equation}
%%%%%%%%%%%%%%%%%
\end{theorem}

This was conjectured in 1973 by D. A. Brannan \cite{Brannan:1974:OCP} who verified the simplest non-trivial case $n=3$, and also showed that the inequality does not hold in general for even $n$. Our aim here is to establish \cref{eq02} for all remaining odd indices $n=5,7,9,\ldots$.

The special case $\beta=1$ has since been settled for all positive odd $n$; see \cite{Barnard:1997:OAC,Barnard:2021:ADP,Deniz:2020:TFS,Jayatilake:2013:BCF,Milcetich:1989:OAC,Szasz:2020:OTB}. In the two-parameter setting, the conjecture was verified on the diagonal $\alpha=\beta\in(0,1)$, and in the triangle $\frac34\le\alpha\le\beta\le1$; see \cite{Cotirla:2024:OTG,Ruscheweyh:2007:OBC}.

In \cite{Dunster:2026:GBC} the general case was recently studied by the present author, where the coefficients were expressed as
%%%%%%%%%%%%%%%%%
\begin{equation}
\label{eq03}
A_n(\alpha,\beta,\omega)
=\frac{(-1)^{n+1}}{\pi}
\omega^n w_{n}(\alpha,\beta,-1/\omega)
\quad (n=1,2,3,\dots),
\end{equation}
%%%%%%%%%%%%%%%%%
in which
%%%%%%%%%%%%%%%%%
\begin{equation}
\label{eq04}
w_{n}(\alpha,\beta,x)
=\frac{1}{n!} \Gamma(n-\alpha)\Gamma(1+\alpha)
\sin(\pi\alpha)\,
{}_2F_1\left(\beta,-n;\alpha-n+1;x\right).
\end{equation}
%%%%%%%%%%%%%%%%%
Here ${}_2F_1$ denotes Gauss's hypergeometric function, defined by \cite[Eq.~15.2.1]{NIST:DLMF}. If we let
%%%%%%%%%%%%%%%%%
\begin{equation}
\label{eq05}
x=-1/\omega=e^{i\phi},
\quad
\phi=\pi-\theta
\quad (\mathrm{mod} \; 2\pi),
\end{equation}
%%%%%%%%%%%%%%%%%
then by the Schwarz reflection principle, Brannan's conjecture is therefore equivalent to
%%%%%%%%%%%%%%%%%
\begin{equation}
\label{eq06}
w_{n}(\alpha,\beta,-1) \ge
\left|w_{n}(\alpha,\beta,e^{i\phi})\right|,
\end{equation}
%%%%%%%%%%%%%%%%%
for $\alpha,\beta\in(0,1]$, $\phi\in[0,\pi]$, and $n=3,5,7,\ldots$. In \cite{Dunster:2026:GBC} it was then shown that:

\begin{theorem}
\label{thm:Watson}
The inequality \cref{eq06} holds for $\alpha,\beta \in (0,1]$, $\phi \in [\phi_0,\pi]$, and $n=5,7,9,\ldots$, where $\phi_0=0.061$.
\end{theorem}

We shall combine this result with two new estimates, which together cover the remaining interval $\phi \in [0,\phi_0]$. To do so, this small interval is split according to the size of 
%%%%%%%%%%%%%%%%%
\begin{equation}
\label{eq07}
\eta:=n\phi.
\end{equation}
%%%%%%%%%%%%%%%%%
In \cref{sec:Meijer-approximations} we construct the required approximations using integrals that can be expressed in terms of the Meijer $G$-function \cite[Sec.~16.17]{NIST:DLMF}. These are then used in \cref{sec:case-I} to prove the following theorem for bounded $\eta$.

\begin{theorem}
\label{thm:main-Meijer}
The inequality \cref{eq06} holds for $\alpha,\beta \in (0,1]$, $\phi \in [0,\phi_0]$, $n=5,7,9,\ldots$, and $\eta \in [0,9]$.
\end{theorem}

The complementary range for unbounded $\eta$ is treated in \cref{sec:case-II}, using a modification of the elementary Watson approximation developed in \cite{Dunster:2026:GBC}, rather than the Meijer $G$-function approximations. We prove the following result.

\begin{theorem}
\label{thm:main-non-Meijer}
The inequality \cref{eq06} holds for $\alpha,\beta \in (0,1]$, $\phi \in (0,\phi_0]$, $n$ odd, and $\eta \in (9,\infty)$.
\end{theorem}

Combining \cref{thm:Watson,thm:main-Meijer,thm:main-non-Meijer}, and using the equivalence \cref{eq03,eq05,eq06}, together with Brannan's proof of the case $n=3$ in \cite{Brannan:1974:OCP}, completes the verification of \cref{thm:Brannan}.

\begin{remark}
The cutoff value $9$ used to separate the two ranges of $\eta$ is a convenient integer choice. The estimates have some tolerance around this value, but larger or smaller integer values are not useful here: numerically we find that larger integer values cause the positivity verification in \cref{sec:case-I} to fail, and the bounds obtained in \cref{sec:case-II} break down when $\eta$ is too small.
\end{remark}

\begin{remark}
Analytical proofs for several bounds used in \cref{sec:case-I,sec:case-II} are given in \cref{sec:lemma-proofs}. Several steps in the proofs of \cref{thm:main-Meijer,thm:main-non-Meijer} reduce the problem to showing positivity of explicitly defined functions on compact parameter domains; the numerical procedures used to obtain these lower bounds are described in \cref{sec:numerics}. The section also explains how the Meijer functions employed in the bounded $\eta$ approximations were evaluated using Maple's built-in routines\footnote{Maple 2025.2. Maplesoft, a division of Waterloo Maple Inc., Waterloo, Ontario.}, and how these computations were checked independently by direct quadrature using \textup{MATLAB}\footnote{MATLAB R2025b. The MathWorks Inc., Natick, MA, USA.}.
\end{remark}

\section{Meijer \texorpdfstring{$G$}{}-function approximations}
\label{sec:Meijer-approximations}
We derive uniform two-term approximations for the integral representations used to estimate $w_n(\alpha,\beta,e^{i\phi})$, valid for $\phi$ close to zero and also at $\phi=0$. In this section no restriction is imposed on the size of $\eta$; the specialisation to the range $\eta \in [0,9]$ will be made in \cref{sec:case-I}.

We shall use the following compound Laplace integral representation from \cite[Lemma 2.1]{Dunster:2026:GBC}, valid for general $x\in\mathbb{C}\setminus[0,\infty)$.
\begin{lemma}
\label{lem:wnLaplace}
Assume $\alpha,\beta \in [0,1]$, $1 \leq n \in \mathbb{N}$, and $x\in\mathbb{C}\setminus[0,\infty)$. Then
%%%%%%%%%%%%%%%%%
\begin{multline}
\label{eq08}
w_{n}(\alpha,\beta,x)=(-1)^{n+1}\sin(\pi\beta)
(-x)^{\,n-\alpha}\int_{0}^{\infty} 
(e^{s}-1)^{-\beta}\,(e^{s}-x)^{\alpha}e^{-ns}\,ds
\\
+ \sin(\pi\alpha)(-x)^{-\beta}
\int_{0}^{\infty} (e^{s}-1)^{\alpha}
\left(e^{s}-1/x\right)^{-\beta}e^{-ns} ds.
\end{multline}
%%%%%%%%%%%%%%%%%
When $\beta=1$ it is understood that the limit of the first term on the RHS applies.
\end{lemma}

We first apply \cref{eq08} to the term $w_n(\alpha,\beta,-1)$ appearing on the LHS of \cref{eq06}, and then use the method of Watson's lemma (\cite[Chap.~3]{Olver:1997:ASF}, \cite[Chap.~2]{Temme:2015:AMF}, \cite[Chap.~5]{Wong:1989:AAI}) to obtain the leading large $n$ approximations with explicit integral remainder terms. For this purpose, define the kernels
%%%%%%%%%%%%%%%%%
\begin{equation}
\label{eq09}
\mathsf{K}_1(\alpha,\beta,s)
=\frac{1}{s^{2}}\left[
\left(\frac{s}{e^{s}-1}\right)^{\beta}
\left(\frac{e^{s}+1}{2}\right)^{\alpha}
-1-\frac{1}{2}(\alpha-\beta)s
\right],
\end{equation}
%%%%%%%%%%%%%%%%%
%%%%%%%%%%%%%%%%%
\begin{equation}
\label{eq10}
\mathsf{K}_2(\alpha,\beta,s)
=\frac{1}{s^{2}}\left[
\left(\frac{e^{s}-1}{s}\right)^{\alpha}
\left(\frac{2}{e^{s}+1}\right)^{\beta}
-1-\frac{1}{2}(\alpha-\beta)s
\right].
\end{equation}
%%%%%%%%%%%%%%%%%
Both $\mathsf K_1(\alpha,\beta,s)$ and $\mathsf K_2(\alpha,\beta,s)$ are $\mathcal O(1)$ as $s\to0$. Indeed, by Taylor expansions the first bracketed quantities in \cref{eq09,eq10} both have the leading Maclaurin expansion $1+\tfrac12(\alpha-\beta)s+\mathcal O(s^2)$, and hence subtracting these first two terms leaves an $\mathcal O(s^2)$ remainder.

Next, on setting $x=-1$ in \cref{eq08}, and using $n$ is odd, gives
%%%%%%%%%%%%%%%%%
\begin{multline}
\label{eq11}
w_n(\alpha,\beta,-1)
=
\sin(\pi\beta)\int_0^\infty
(e^s-1)^{-\beta}(e^s+1)^\alpha e^{-ns}\,d s
\\
+
\sin(\pi\alpha)\int_0^\infty
(e^s-1)^\alpha(e^s+1)^{-\beta}e^{-ns}\,d s .
\end{multline}
%%%%%%%%%%%%%%%%%
By the definitions \cref{eq09,eq10}, the two integrands may be written as
%%%%%%%%%%%%%%%%%
\begin{equation}
\label{eq12}
(e^s-1)^{-\beta}(e^s+1)^\alpha
=
2^\alpha s^{-\beta}
\left\{
1+\tfrac12(\alpha-\beta)s+s^2\mathsf K_1(\alpha,\beta,s)
\right\},
\end{equation}
%%%%%%%%%%%%%%%%%
and
%%%%%%%%%%%%%%%%%
\begin{equation}
\label{eq13}
(e^s-1)^\alpha(e^s+1)^{-\beta}
=
2^{-\beta}s^\alpha
\left\{
1+\tfrac12(\alpha-\beta)s+s^2\mathsf K_2(\alpha,\beta,s)
\right\}.
\end{equation}
%%%%%%%%%%%%%%%%%
Substituting \cref{eq12,eq13} into \cref{eq11} and using the elementary gamma-integral identity
%%%%%%%%%%%%%%%%%
\begin{equation}
\label{eq14}
\int_0^\infty t^{\mu}e^{-\lambda t} dt
=
\lambda^{-1-\mu}\Gamma(1+\mu)
\quad (\Re(\mu)>-1,\, \Re(\lambda)>0),
\end{equation}
%%%%%%%%%%%%%%%%%
gives, for odd positive integers $n$,
%%%%%%%%%%%%%%%%%
\begin{multline}
\label{eq15}
w_n(\alpha,\beta,-1)
=\sin(\pi\beta)\,2^\alpha\Gamma(1-\beta)\,n^{-1+\beta}
\left\{1+\frac{(\alpha-\beta)(1-\beta)}{2n}\right\}
\\
+\sin(\pi\alpha)\,2^{-\beta}\Gamma(1+\alpha)\,n^{-1-\alpha}
\left\{1+\frac{(\alpha-\beta)(1+\alpha)}{2n}\right\}
\\
+\sin(\pi\beta)\,2^\alpha\int_0^\infty s^{2-\beta}\mathsf{K}_1(\alpha,\beta,s)e^{-ns}\,d s
\\
+\sin(\pi\alpha)\,2^{-\beta}\int_0^\infty s^{2+\alpha}\mathsf{K}_2(\alpha,\beta,s)e^{-ns}\,d s.
\end{multline}
%%%%%%%%%%%%%%%%%
This is an asymptotic approximation in the sense that the first two terms dominate when $n \to \infty$.

Our main task is to obtain uniform asymptotic approximations for $w_{n}(\alpha,\beta,e^{i\phi})$ in a similar way, but which are valid for $\phi$ close to zero, including $\phi=0$. Although \cref{eq15} provides a Watson-type expansion for $w_n(\alpha,\beta,-1)$, the corresponding elementary Watson approximation for $w_n(\alpha,\beta,e^{i\phi})$ is not uniform as $\phi\to0$; see \cite{Dunster:2026:GBC}.

To this end, define
%%%%%%%%%%%%%%%%%
\begin{equation}
\label{eq16}
J_1(\alpha,\beta,\phi,n)=\int_0^\infty (e^s-1)^{-\beta}
R(s,\phi)^{\alpha}\,e^{-ns}\,d s,
\end{equation}
%%%%%%%%%%%%%%%%%
and
%%%%%%%%%%%%%%%%%
\begin{equation}
\label{eq17}
J_2(\alpha,\beta,\phi,n)=\int_0^\infty (e^s-1)^{\alpha}
R(s,\phi)^{-\beta}\,e^{-ns}\,d s,
\end{equation}
%%%%%%%%%%%%%%%%%
where
%%%%%%%%%%%%%%%%%
\begin{equation}
\label{eq18}
R(s,\phi)=\left|e^{s}-e^{\pm i \phi}\right|
=\sqrt{e^{2s}+1-2e^{s}\cos(\phi)}.
\end{equation}
%%%%%%%%%%%%%%%%%
Then, on taking absolute values in \cref{eq08} with $x=e^{i\phi}$, the following inequality given in \cite{Dunster:2026:GBC} is obtained, which is the foundation for the estimates in the rest of the paper: 
%%%%%%%%%%%%%%%%%
\begin{equation}
\label{eq19}
\bigl|w_n(\alpha,\beta,e^{i \phi})\bigr|
\le 
\sin(\pi\beta)J_1(\alpha,\beta,\phi,n)
+\sin(\pi\alpha)J_2(\alpha,\beta,\phi,n),
\end{equation}
%%%%%%%%%%%%%%%%%
valid for $\alpha,\beta\in(0,1]$, positive integers $n$, and $\phi\in[0,\pi]$. 

\begin{remark}
When $\beta=1$, the product $\sin(\pi\beta)J_1(\alpha,\beta,\phi,n)$ in \cref{eq19} is understood by continuous extension from $0<\beta<1$. More generally, throughout the paper, products containing removable endpoint singularities are interpreted by continuous extension from the interior parameter range. In addition, although $\phi=0$ (i.e. $x=1$) is excluded from \cref{lem:wnLaplace}, the inequality \cref{eq19} is understood as the limit as $\phi\to 0^+$. This limit is legitimate since $w_n(\alpha,\beta,x)$ is continuous at $x=1$, and, for $\alpha,\beta \in (0,1)$, the integrands of \cref{eq16,eq17} converge locally to integrable functions near $s=0$, with behaviour $\mathcal O(s^{\alpha-\beta})$. The endpoint cases are then obtained by continuous extension.
\end{remark}

We now approximate the integrands of \cref{eq16,eq17} for small $s$, uniformly for $\phi \in [0,\phi_{0}]$ ($\phi_{0}=0.061$). With this in mind, define
%%%%%%%%%%%%%%%%%
\begin{equation}
\label{eq20}
\Upsilon(s,\rho)
=R(s,\phi)^2
=(e^s-1)^2+\rho^2\,e^s,
\end{equation}
%%%%%%%%%%%%%%%%%
where for convenience we use the real variable
%%%%%%%%%%%%%%%%%
\begin{equation}
\label{eq21}
\rho:=R(0,\phi)=2\sin\left(\tfrac{1}{2}\phi\right)=\phi
+\mathcal{O}(\phi^{3})
\quad (\phi \to 0).
\end{equation}
%%%%%%%%%%%%%%%%%

Consider the integrand of \cref{eq16} first. We seek an approximation for small $s$ that is uniform for $\phi \ge 0$ ($\rho \ge 0$). Based on the approximation $\Upsilon(s,\rho)=s^{2}+\rho^{2}+O(s^{3})+O(\rho^{2}s)$ for small $s$ and $\rho$, we divide $\Upsilon(s,\rho)$ by the leading term $s^2+\rho^2$, and for fixed $\rho>0$ it can readily be shown that
%%%%%%%%%%%%%%%%%
\begin{equation}
\label{eq22}
\frac{\Upsilon(s,\rho)}
{s^2+\rho^2}
=
1+s+\frac{1}{2}s^2+\frac{1}{6}s^3
+\frac{2+\rho^2}{24\rho^2}s^4
+\mathcal{O}\left(s^{5}\right)
\quad (s \to 0).
\end{equation}
%%%%%%%%%%%%%%%%%
We shall use the first two terms in this and claim that these provide a uniform approximation for bounded $\rho \ge 0$; it can be seen that this is certainly not true if we take more than three terms. 

To verify this, define $e(s,\rho)$ by
%%%%%%%%%%%%%%%%%
\begin{equation}
\label{eq23}
\frac{\Upsilon(s,\rho)}
{s^2+\rho^2}
=
1+s+e(s,\rho) s^2.
\end{equation}
%%%%%%%%%%%%%%%%%
Now using $(e^s-1)^2-s^2-s^3=\frac{7}{12}s^{4}+\mathcal{O}(s^{5})$ and $e^s-1-s=\frac{1}{2}s^{2}+\mathcal{O}(s^{3})$ as $s \to 0$, we observe from \cref{eq20,eq23} that
%%%%%%%%%%%%%%%%%
\begin{multline}
\label{eq24}
e(s,\rho)
=\frac{(e^s-1)^2-s^2-s^3}{s^{2}(s^{2}+\rho^2)}
+\rho^2\frac{e^s-1-s}{s^{2}(s^{2}+\rho^2)}
\\
=\left\{\frac{7s^{2}}{12(s^{2}+\rho^2)}
+\frac{\rho^2}{2(s^{2}+\rho^2)}\right\}
\left\{1+\mathcal{O}(s)\right\}
=\mathcal{O}(1),
\end{multline}
%%%%%%%%%%%%%%%%%
as $s \to 0$ uniformly for bounded $\rho \ge 0$, and in particular for $\phi\in[0,\phi_0]$ ($\rho\in[0,2\sin(\tfrac{1}{2}\phi_0)]$). Hence from \cref{eq23,eq24} we obtain our desired approximation
%%%%%%%%%%%%%%%%%
\begin{equation}
\label{eq25}
\left(\frac{\Upsilon(s,\rho)}
{s^2+\rho^2}\right)^{\alpha/2}
=
1+\frac{1}{2}\alpha s+\mathcal{O}(s^2)
\quad (s\to 0,\, \phi\in[0,\phi_0]).
\end{equation}
%%%%%%%%%%%%%%%%%

Next
%%%%%%%%%%%%%%%%%
\begin{equation}
\label{eq26}
\left(\frac{s}{e^s-1}\right)^\beta
=1-\frac{1}{2}\beta s+\mathcal{O}(s^2)
\quad (s\to0),
\end{equation}
%%%%%%%%%%%%%%%%%
and therefore
%%%%%%%%%%%%%%%%%
\begin{equation}
\label{eq27}
\left(\frac{s}{e^{s}-1}\right)^{\beta}
\left(\frac{\Upsilon(s,\rho)}{s^2+\rho^2}\right)^{\alpha/2}
=
p(\alpha,\beta,s)
+s^2\widetilde{\mathsf{K}}_1(\alpha,\beta,\rho,s),
\end{equation}
%%%%%%%%%%%%%%%%%
where
%%%%%%%%%%%%%%%%%
\begin{equation}
\label{eq28}
p(\alpha,\beta,s)
=1+\tfrac{1}{2}(\alpha-\beta)s,
\end{equation}
%%%%%%%%%%%%%%%%%
and 
%%%%%%%%%%%%%%%%%
\begin{equation}
\label{eq29}
\widetilde{\mathsf{K}}_1(\alpha,\beta,\rho,s)
=
\frac{1}{s^{2}}\left[
\left(\frac{s}{e^{s}-1}\right)^{\beta}
\left(\frac{\Upsilon(s,\rho)}{s^2+\rho^2}\right)^{\alpha/2}
-p(\alpha,\beta,s)
\right]
=\mathcal{O}(1),
\end{equation}
%%%%%%%%%%%%%%%%%
as $s \to 0$ uniformly for $\phi\in[0,\phi_0]$. Then, for \cref{eq16}, we have from \cref{eq20,eq28,eq27}
%%%%%%%%%%%%%%%%%
\begin{multline}
\label{eq30}
(e^s-1)^{-\beta}R(s,\phi)^\alpha
=(e^s-1)^{-\beta}\Upsilon(s,\rho)^{\alpha/2}
=s^{-\beta}(s^2+\rho^2)^{\alpha/2}
\\
+\tfrac{1}{2}(\alpha-\beta)s^{1-\beta}(s^2+\rho^2)^{\alpha/2}
+s^{2-\beta}(s^2+\rho^2)^{\alpha/2}
\widetilde{\mathsf{K}}_1(\alpha,\beta,\rho,s).
\end{multline}
%%%%%%%%%%%%%%%%%

Similarly, with \cref{eq17} in mind, it is straightforward to show that
%%%%%%%%%%%%%%%%%
\begin{equation}
\label{eq31}
\left(\frac{e^{s}-1}{s}\right)^{\alpha}
\left(\frac{s^2+\rho^2}{\Upsilon(s,\rho)}\right)^{\beta/2}
=
p(\alpha,\beta,s)
+s^2\widetilde{\mathsf{K}}_2(\alpha,\beta,\rho,s),
\end{equation}
%%%%%%%%%%%%%%%%%
where
%%%%%%%%%%%%%%%%%
\begin{equation}
\label{eq32}
\widetilde{\mathsf{K}}_2(\alpha,\beta,\rho,s)
=
\frac{1}{s^{2}}\left[
\left(\frac{e^{s}-1}{s}\right)^{\alpha}
\left(\frac{s^2+\rho^2}{\Upsilon(s,\rho)}\right)^{\beta/2}
-p(\alpha,\beta,s)
\right]
=\mathcal{O}(1),
\end{equation}
%%%%%%%%%%%%%%%%%
as $s \to 0$, uniformly for $\phi\in[0,\phi_0]$. Therefore from \cref{eq20,eq28,eq31}
%%%%%%%%%%%%%%%%%
\begin{multline}
\label{eq33}
(e^s-1)^{\alpha}R(s,\phi)^{-\beta}
=(e^s-1)^{\alpha}\Upsilon(s,\rho)^{-\beta/2}
=s^\alpha(s^2+\rho^2)^{-\beta/2}
\\
+\tfrac{1}{2}(\alpha-\beta)s^{1+\alpha}
(s^2+\rho^2)^{-\beta/2}
+s^{2+\alpha}(s^2+\rho^2)^{-\beta/2}
\widetilde{\mathsf{K}}_2(\alpha,\beta,\rho,s).
\end{multline}
%%%%%%%%%%%%%%%%%
Thus from \cref{eq16,eq17,eq30,eq33} we arrive at the uniform approximations
%%%%%%%%%%%%%%%%%
\begin{multline}
\label{eq34}
J_1(\alpha,\beta,\phi,n)
=
\hat M_{1,0}(\alpha,\beta,\rho,n)
+\tfrac{1}{2}(\alpha-\beta)
\hat M_{1,1}(\alpha,\beta,\rho,n)
\\
+\int_0^\infty s^{2-\beta}(s^2+\rho^2)^{\alpha/2}
\widetilde{\mathsf{K}}_1(\alpha,\beta,\rho,s)
e^{-ns}\,d s,
\end{multline}
%%%%%%%%%%%%%%%%%
and
%%%%%%%%%%%%%%%%%
\begin{multline}
\label{eq35}
J_2(\alpha,\beta,\phi,n)
=
\hat M_{2,0}(\alpha,\beta,\rho,n)
+\tfrac{1}{2}(\alpha-\beta)
\hat M_{2,1}(\alpha,\beta,\rho,n)
\\
+\int_0^\infty s^{2+\alpha}(s^2+\rho^2)^{-\beta/2}
\widetilde{\mathsf{K}}_2(\alpha,\beta,\rho,s)e^{-ns}\,d s,
\end{multline}
%%%%%%%%%%%%%%%%%
where
%%%%%%%%%%%%%%%%%
\begin{equation}
\label{eq36}
\hat M_{1,k}(\alpha,\beta,\rho,n)
=
\int_0^\infty s^{k-\beta}
(s^2+\rho^2)^{\alpha/2}e^{-ns}\,d s
\quad (k=0,1),
\end{equation}
%%%%%%%%%%%%%%%%%
and
%%%%%%%%%%%%%%%%%
\begin{equation}
\label{eq37}
\hat M_{2,k}(\alpha,\beta,\rho,n)
=
\int_0^\infty s^{k+\alpha}
(s^2+\rho^2)^{-\beta/2}e^{-ns}\,d s
\quad (k=0,1).
\end{equation}
%%%%%%%%%%%%%%%%%

We find it convenient to re-express these integrals. Thus, let
%%%%%%%%%%%%%%%%%
\begin{equation}
\label{eq38}
\zeta:=n^2\rho^2
=4 n^2 \sin^2\left(\tfrac{1}{2}\phi\right),
\end{equation}
%%%%%%%%%%%%%%%%%
and then, after the change of variable $t=ns$, they become
%%%%%%%%%%%%%%%%%
\begin{equation}
\label{eq39}
\hat M_{j,k}(\alpha,\beta,\rho,n)
=
n^{-1-k-\alpha+\beta}M_{j,k}(\alpha,\beta,\zeta)
\quad (j=1,2,\,k=0,1),
\end{equation}
%%%%%%%%%%%%%%%%%
where for $\zeta \ge 0$ and $\alpha,\beta \in [0,1]$ (unless otherwise noted)
%%%%%%%%%%%%%%%%%
\begin{equation}
\label{eq40}
M_{1,0}(\alpha,\beta,\zeta)
=
\int_0^\infty t^{-\beta}(t^2+\zeta)^{\alpha/2}e^{-t}\,d t
\quad \left(\beta<1\text{ or }\zeta=0<\alpha\right),
\end{equation}
%%%%%%%%%%%%%%%%%
%%%%%%%%%%%%%%%%%
\begin{equation}
\label{eq41}
M_{1,1}(\alpha,\beta,\zeta)
=
\int_0^\infty t^{1-\beta}(t^2+\zeta)^{\alpha/2}e^{-t}\,d t,
\end{equation}
%%%%%%%%%%%%%%%%%
%%%%%%%%%%%%%%%%%
\begin{equation}
\label{eq42}
M_{2,0}(\alpha,\beta,\zeta)
=
\int_0^\infty t^{\alpha}(t^2+\zeta)^{-\beta/2}e^{-t}\,d t
\quad \left((\alpha,\beta,\zeta)\ne(0,1,0)\right),
\end{equation}
%%%%%%%%%%%%%%%%%
and
%%%%%%%%%%%%%%%%%
\begin{equation}
\label{eq43}
M_{2,1}(\alpha,\beta,\zeta)
=
\int_0^\infty t^{1+\alpha}(t^2+\zeta)^{-\beta/2}e^{-t}\,d t.
\end{equation}
%%%%%%%%%%%%%%%%%
From \cref{eq34,eq35,eq39} we arrive at our desired representations.
\begin{lemma}
\label{lem:J-Meijer}
Let $\alpha,\beta\in(0,1)$, $n>1$, and $\phi\in[0,\pi]$. Then the following representations hold, with endpoint $\alpha,\beta$ values permitted subject to convergence of the defining integrals for $M_{j,k}(\alpha,\beta,\zeta)$:
%%%%%%%%%%%%%%%%%
\begin{multline}
\label{eq44}
J_1(\alpha,\beta,\phi,n)
=
n^{-1-\alpha+\beta}M_{1,0}(\alpha,\beta,\zeta)
+\tfrac{1}{2}(\alpha-\beta)n^{-2-\alpha+\beta}M_{1,1}(\alpha,\beta,\zeta)
\\
+\int_0^\infty s^{2-\beta}(s^2+\rho^2)^{\alpha/2}
\widetilde{\mathsf{K}}_1(\alpha,\beta,\rho,s)
e^{-ns}\,d s,
\end{multline}
%%%%%%%%%%%%%%%%%
and
%%%%%%%%%%%%%%%%%
\begin{multline}
\label{eq45}
J_2(\alpha,\beta,\phi,n)
=
n^{-1-\alpha+\beta}M_{2,0}(\alpha,\beta,\zeta)
+\tfrac{1}{2}(\alpha-\beta)n^{-2-\alpha+\beta}M_{2,1}(\alpha,\beta,\zeta)
\\
+\int_0^\infty s^{2+\alpha}(s^2+\rho^2)^{-\beta/2}
\widetilde{\mathsf{K}}_2(\alpha,\beta,\rho,s)
e^{-ns}\,d s.
\end{multline}
%%%%%%%%%%%%%%%%%
\end{lemma}

From \cref{eq14,eq40,eq41,eq42,eq43} it is straightforward to show that
%%%%%%%%%%%%%%%%%
\begin{equation}
\label{eq46}
M_{j,k}(\alpha,\beta,0)
=\Gamma(1+k+\alpha-\beta)
\quad
(j=1,2,\, k=0,1).
\end{equation}
%%%%%%%%%%%%%%%%%
In our application $\zeta$ will be bounded above by 81, but the following slowly-varying behaviour as $\zeta\to\infty$ is also worth noting: 
%%%%%%%%%%%%%%%%%
\begin{equation}
\label{eq47}
M_{1,k}(\alpha,\beta,\zeta)
=
\Gamma(1+k-\beta)\,\zeta^{\alpha/2}
+\mathcal{O}(\zeta^{\alpha/2-1}),
\end{equation}
%%%%%%%%%%%%%%%%%
and
%%%%%%%%%%%%%%%%%
\begin{equation}
\label{eq48}
M_{2,k}(\alpha,\beta,\zeta)
=
\Gamma(1+k+\alpha)\,\zeta^{-\beta/2}
+\mathcal{O}(\zeta^{-\beta/2-1}),
\end{equation}
%%%%%%%%%%%%%%%%%
for $k=0,1$. Using \cref{eq14}, we also observe the following special cases:
%%%%%%%%%%%%%%%%%
\begin{equation}
\label{eq49}
M_{1,k}(0,\beta,\zeta)
=\Gamma(1+k-\beta),
\end{equation}
%%%%%%%%%%%%%%%%%
and
%%%%%%%%%%%%%%%%%
\begin{equation}
\label{eq50}
M_{2,k}(\alpha,0,\zeta)
=\Gamma(1+k+\alpha),
\end{equation}
%%%%%%%%%%%%%%%%%
for $k=0,1$.

The integrals \cref{eq40,eq41,eq42,eq43} may be expressed in terms of known special functions. In particular, \cite[p.~369, formula~2.7]{Erdelyi:1954:TI2} gives
%%%%%%%%%%%%%%%%%
\begin{equation}
\label{eq51}
M_{1,k}(\alpha,\beta,\zeta)
=
\frac{\zeta^{(1+k+\alpha-\beta)/2}}
{2\sqrt{\pi}\,\Gamma\left(-\tfrac12 \alpha\right)}
\,G^{3,1}_{1,3}\left(
\frac{\zeta}{4}\,\middle|\,
\begin{matrix}
\frac{1}{2}(1-k+\beta)\\[1mm]
-\frac{1}{2}(1+k+\alpha-\beta),\,0,\,\frac12
\end{matrix}
\right)
\quad (k=0,1),
\end{equation}
%%%%%%%%%%%%%%%%%
and
%%%%%%%%%%%%%%%%%
\begin{equation}
\label{eq52}
M_{2,k}(\alpha,\beta,\zeta)
=
\frac{\zeta^{(1+k+\alpha-\beta)/2}}
{2\sqrt{\pi}\,\Gamma\left(\frac12\beta\right)}
\,G^{3,1}_{1,3}\left(
\frac{\zeta}{4}\,\middle|\,
\begin{matrix}
\frac{1}{2}(1-k-\alpha)\\[1mm]
-\frac{1}{2}(1+k+\alpha-\beta),\,0,\,\frac12
\end{matrix}
\right)
\quad (k=0,1),
\end{equation}
%%%%%%%%%%%%%%%%%
where the special case of Meijer $G$-function is given here by
%%%%%%%%%%%%%%%%%
\begin{equation}
\label{eq53}
G_{1,3}^{\,3,1}\left(z\,\middle|
\,\begin{matrix} a\\ b_1,b_2,b_3 \end{matrix}\right)
=\frac{1}{2\pi i}\int_\mathcal{L}
\Gamma(b_1-s)\Gamma(b_2-s)\Gamma(b_3-s)
\Gamma(1-a+s)\,z^{\,s}\,d s,
\end{equation}
%%%%%%%%%%%%%%%%%
with $\mathcal{L}$ being a contour from $-i\infty$ to $+i\infty$ that separates the poles of $\Gamma(b_1-s)$, $\Gamma(b_2-s)$, and $\Gamma(b_3-s)$ from those of $\Gamma(1-a+s)$. It also satisfies the third order linear differential equation
%%%%%%%%%%%%%%%%%
\begin{equation}
\label{eq54}
\left\{ \left(z\frac{d}{dz}-b_1\right)
\left(z\frac{d}{dz}-b_2\right)\left(z\frac{d}{dz}
-b_3\right)+z\left(z\frac{d}{dz}
-a+1\right) \right\}y(z)=0.
\end{equation}
%%%%%%%%%%%%%%%%%
We also note that the integrals \cref{eq40,eq41,eq42,eq43} can alternatively be expressed as sums of three generalized hypergeometric functions ${}_1F_2$; see \cite[p.~326, formula~(8)]{Prudnikov:1986:IS1}.

In the present application the Meijer $G$-functions initially arise with bounded real arguments $\frac{1}{4}\zeta\in[0,\frac{81}{4}]$ and small real parameters. In this regime they are readily computable, since they are slowly varying, with no oscillatory complex arguments or branch-cut difficulties, and their defining integrals \cref{eq40,eq41,eq42,eq43} have exponentially decaying tails. As we shall show in \cref{sec:numerics}, the numerical verification can in fact be reduced to Meijer $G$-function evaluations only at the endpoints $\zeta=0$ and $\zeta=81$. At $\zeta=0$, the values are given explicitly by \cref{eq46}, so that numerical evaluation of the Meijer $G$-functions is only required at $\zeta=81$.

\section{\texorpdfstring{Proof of \cref{thm:main-Meijer}}{Proof of Case I Theorem}}
\label{sec:case-I}

We first record a general lower bound, obtained from \cref{eq15,eq19,eq44,eq45}.

\begin{lemma}
\label{lem:Meijer-bound}
Let $\alpha,\beta\in(0,1]$, $n=1,3,5,\ldots$, and $\phi \in [0,\pi]$. Then
%%%%%%%%%%%%%%%%%
\begin{equation}
\label{eq55}
w_n(\alpha,\beta,-1)-\bigl|w_n(\alpha,\beta,
e^{i \phi})\bigr|
\ge
H_n(\alpha,\beta,\zeta)
+
\int_0^\infty \mathsf{L}(\alpha,\beta,\rho,s)e^{-ns}\,d s,
\end{equation}
%%%%%%%%%%%%%%%%%
where
%%%%%%%%%%%%%%%%%
\begin{multline}
\label{eq56}
H_n(\alpha,\beta,\zeta)
=
\sin(\pi\beta)\Biggl[
2^\alpha\Gamma(1-\beta)\,n^{-1+\beta}
\left\{1+\frac{(\alpha-\beta)(1-\beta)}{2n}\right\}
\\
-\,n^{-1-\alpha+\beta}M_{1,0}(\alpha,\beta,\zeta)
-\frac{1}{2}(\alpha-\beta)n^{-2-\alpha+\beta}M_{1,1}(\alpha,\beta,\zeta)
\Biggr]
\\
+\sin(\pi\alpha)\Biggl[
2^{-\beta}\Gamma(1+\alpha)\,n^{-1-\alpha}
\left\{1+\frac{(\alpha-\beta)(1+\alpha)}{2n}\right\}
\\
-\,n^{-1-\alpha+\beta}M_{2,0}(\alpha,\beta,\zeta)
-\frac{1}{2}(\alpha-\beta)n^{-2-\alpha+\beta}M_{2,1}(\alpha,\beta,\zeta)
\Biggr],
\end{multline}
%%%%%%%%%%%%%%%%%
and
%%%%%%%%%%%%%%%%%
\begin{multline}
\label{eq57}
\mathsf{L}(\alpha,\beta,\rho,s)
=\sin(\pi\beta)\Bigl[
2^\alpha s^{2-\beta}\mathsf{K}_1(\alpha,\beta,s)
-s^{2-\beta}(s^2+\rho^2)^{\alpha/2}\widetilde{\mathsf{K}}_1(\alpha,\beta,\rho,s)
\Bigr]
\\
+\sin(\pi\alpha)\Bigl[
2^{-\beta}s^{2+\alpha}\mathsf{K}_2(\alpha,\beta,s)
-s^{2+\alpha}(s^2+\rho^2)^{-\beta/2}\widetilde{\mathsf{K}}_2(\alpha,\beta,\rho,s)
\Bigr].
\end{multline}
%%%%%%%%%%%%%%%%%
Here $\zeta=n^2\rho^2$, $\rho=2\sin(\frac12 \phi)$, with $M_{1,0}(\alpha,\beta,\zeta)$, $M_{1,1}(\alpha,\beta,\zeta)$, $M_{2,0}(\alpha,\beta,\zeta)$, \linebreak $M_{2,1}(\alpha,\beta,\zeta)$, $\mathsf{K}_1(\alpha,\beta,s)$, $\mathsf{K}_2(\alpha,\beta,s)$, $\widetilde{\mathsf{K}}_1(\alpha,\beta,\rho,s)$, and $\widetilde{\mathsf{K}}_2(\alpha,\beta,\rho,s)$ given by \cref{eq40,eq41,eq42,eq43,eq09,eq10,eq29,eq32}, respectively.
\end{lemma}

\begin{remark}
The convention concerning continuous extension at endpoint parameter values, stated after \cref{eq19}, applies in particular to products such as $\sin(\pi\beta)\Gamma(1-\beta)$ and $\sin(\pi\beta)M_{1,k}(\alpha,\beta,\zeta)$.
\end{remark}

For the remainder of this section we restrict to the theorem range $\phi \in [0,\phi_0]$ and $\eta \in [0,9]$, and record a number of preliminary results to obtain the desired positivity of the RHS of \cref{eq55} under these constraints. The first is a lower bound for the remainder term, with the proof given in \cref{sec:lemma-proofs}.

\begin{lemma}
\label{lem:R-bound}
For $\alpha,\beta \in (0,1]$, $\phi \in [0,\phi_0]$ and $n>1$
%%%%%%%%%%%%%%%%%
\begin{equation}
\label{eq58}
\int_0^\infty \mathsf{L}(\alpha,\beta,\rho,s)e^{-ns}\,d s
>
-0.0254\,\Gamma(2-\beta)
\sigma(\alpha,\beta)(n-1)^{-2+\beta},
\end{equation}
%%%%%%%%%%%%%%%%%
where
%%%%%%%%%%%%%%%%%
\begin{equation}
\label{eq59}
\sigma(\alpha,\beta)
=\frac{\alpha\beta 
(\alpha+\beta)
\left\{\sin(\pi\alpha)+\sin(\pi\beta)\right\}}
{\alpha(1-\alpha)+\beta(1-\beta)}.
\end{equation}
%%%%%%%%%%%%%%%%%
\end{lemma}

\begin{remark}
\label{rem:sigma-tildeL}
$\sigma(\alpha,\beta)$ is bounded for $\alpha, \beta \in [0,1]$, and vanishes as $(\alpha,\beta)$ approaches the edges $\alpha=0$ and $\beta=0$. In particular
%%%%%%%%%%%%%%%%%
\begin{equation}
\label{eq60}
\sigma(\alpha,\beta)
=
\frac{\beta\sin(\pi\beta)}{1-\beta}\,\alpha
+\mathcal{O}(\alpha^2)
\quad (\alpha \to 0, \, \beta \in [0,1)),
\end{equation}
%%%%%%%%%%%%%%%%%
and
%%%%%%%%%%%%%%%%%
\begin{equation}
\label{eq61}
\sigma(\alpha,\beta)
=
\frac{\alpha\sin(\pi\alpha)}{1-\alpha}\,\beta
+\mathcal{O}(\beta^2)
\quad (\beta \to 0, \, \alpha \in [0,1)),
\end{equation}
%%%%%%%%%%%%%%%%%
with the limits of \cref{eq60,eq61} applying when $\beta \to 1$ and $\alpha \to 1$, respectively. In addition
%%%%%%%%%%%%%%%%%
\begin{equation}
\label{eq62}
\sigma(\alpha,\beta)
\sim \pi\alpha\beta(\alpha+\beta)
\quad ((\alpha,\beta) \to (0,0)),
\end{equation}
%%%%%%%%%%%%%%%%%
and $\sigma(1,1)=2\pi$.
\end{remark}

Next, in conjunction with \cref{eq58}, is a numerical verification when $\alpha,\beta\in (0,1]$ and $\zeta\in[0,81]$ that the RHS of \cref{eq55} is positive for $n=5$ (the smallest value under consideration).

\begin{proposition}
\label{prop:scriptH5}
Define
%%%%%%%%%%%%%%%%%
\begin{equation}
\label{eq63}
\mathcal H_5(\alpha,\beta,\zeta)
=\frac{H_5(\alpha,\beta,\zeta)-0.0254\,\Gamma(2-\beta)\sigma(\alpha,\beta)4^{-2+\beta}}{\alpha\beta(\alpha+\beta)},
\end{equation}
%%%%%%%%%%%%%%%%%
where $H_n(\alpha,\beta,\zeta)$ and $\sigma(\alpha,\beta)$ are given by \cref{eq56,eq59}. Then, for $\alpha,\beta\in[0,1]$ and $\zeta\in[0,81]$,
%%%%%%%%%%%%%%%%%
\begin{equation}
\label{eq64}
\mathcal H_5(\alpha,\beta,\zeta)\ge 0.0388137394\cdots.
\end{equation}
%%%%%%%%%%%%%%%%%
\end{proposition}

The next lemma, whose proof is also deferred to \cref{sec:lemma-proofs}, will be used to extend the above positivity to $n>5$. In this, in order to obtain sufficiently sharp bounds, we use the elementary inequality
%%%%%%%%%%%%%%%%%
\begin{equation}
\label{eq65}
-an^{-q}\ge -[a]_+m^{-q},
\quad n\ge m >0,\ q>0,
\end{equation}
%%%%%%%%%%%%%%%%%
where
%%%%%%%%%%%%%%%%%
\begin{equation}
\label{eq66}
[a]_+=\max\{a,0\},
\end{equation}
%%%%%%%%%%%%%%%%%
as opposed to the cruder bound $-an^{-q}\ge -|a|m^{-q}$.

\begin{lemma}
\label{lem:n-derivative}
Let
%%%%%%%%%%%%%%%%%
\begin{equation}
\label{eq67}
h_1(\alpha,\beta)=2^\alpha \sin(\pi\beta)\Gamma(1-\beta),
\end{equation}
%%%%%%%%%%%%%%%%%
%%%%%%%%%%%%%%%%%
\begin{equation}
\label{eq68}
h_2(\alpha,\beta)=\sin(\pi\alpha)\,2^{-\beta}\Gamma(1+\alpha)(1-\alpha-\beta),
\end{equation}
%%%%%%%%%%%%%%%%%
%%%%%%%%%%%%%%%%%
\begin{equation}
\label{eq69}
h_3(\alpha,\beta,\zeta)
=-\sin(\pi\beta)(1-\alpha)M_{1,0}(\alpha,\beta,\zeta)
-\sin(\pi\alpha)(1-\alpha)M_{2,0}(\alpha,\beta,\zeta),
\end{equation}
%%%%%%%%%%%%%%%%%
and
%%%%%%%%%%%%%%%%%
\begin{multline}
\label{eq70}
h_4(\alpha,\beta,\zeta)
=
\alpha(\beta-\alpha)\sin(\pi\beta)M_{1,1}(\alpha,\beta,\zeta)
\\
+\sin(\pi\alpha)
\left\{
\alpha(\beta-\alpha)M_{2,1}(\alpha,\beta,\zeta)
+\Gamma(2+\alpha)[\alpha^2-\beta^2]_+10^{-\beta}
\right\}.
\end{multline}
%%%%%%%%%%%%%%%%%
Then for $\alpha,\beta \in (0,1]$, $\zeta \ge 0$, and $n \ge 5$
%%%%%%%%%%%%%%%%%
\begin{multline}
\label{eq71}
n^{\alpha}\frac{\partial}
{\partial n}\left[n^{2-\beta}
\left\{H_n(\alpha,\beta,\zeta)
-0.0254\,\sigma(\alpha,\beta)
\Gamma(2-\beta)\,(n-1)^{-2+\beta} \right\}\right]
\\
\ge h_1(\alpha,\beta)5^\alpha+h_2(\alpha,\beta)5^{-\beta}
+h_3(\alpha,\beta,\zeta)
-\tfrac{1}{10}[h_4(\alpha,\beta,\zeta)]_+,
\end{multline}
%%%%%%%%%%%%%%%%%
where in the $n$ differentiation $\zeta$ is treated as an independent parameter.
\end{lemma}

\begin{remark}
Here and elsewhere, treating $\zeta$ as independent of $n$ when differentiating with respect to $n$ is legitimate because the derivative estimate is proved for each fixed value of $\zeta$, uniformly in $n$. In the subsequent proof this fixed $\zeta$ estimate is used uniformly for $\zeta\in[0,81]$.
\end{remark}

The non-negativity of the RHS of \cref{eq71} is numerically verified as follows.

\begin{proposition}
\label{prop:Q5}
Define
%%%%%%%%%%%%%%%%%
\begin{equation}
\label{eq72}
\mathcal Q_5(\alpha,\beta,\zeta)
=
\frac{
h_1(\alpha,\beta)5^\alpha+h_2(\alpha,\beta)5^{-\beta}
+h_3(\alpha,\beta,\zeta)
-\tfrac{1}{10}[h_4(\alpha,\beta,\zeta)]_+
}
{\alpha\beta(\alpha+\beta)}.
\end{equation}
%%%%%%%%%%%%%%%%%
Then, for $\alpha,\beta\in[0,1]$ and $\zeta\in[0,81]$,
%%%%%%%%%%%%%%%%%
\begin{equation}
\label{eq73}
\mathcal Q_5(\alpha,\beta,\zeta)\ge 3.1274558648\cdots.
\end{equation}
%%%%%%%%%%%%%%%%%
\end{proposition}

With these preliminary results in place, we now verify the main result of this section.

\begin{proof}[Proof of \cref{thm:main-Meijer}]
Set
%%%%%%%%%%%%%%%%%
\begin{equation}
\label{eq74}
\mathcal F_n(\alpha,\beta,\zeta)
=
n^{2-\beta}\left\{
H_n(\alpha,\beta,\zeta)
-0.0254\,\sigma(\alpha,\beta)
\Gamma(2-\beta)(n-1)^{-2+\beta}
\right\},
\end{equation}
%%%%%%%%%%%%%%%%%
where $H_n(\alpha,\beta,\zeta)$ is defined by \cref{eq56}. Since $\zeta=4n^2\sin^2(\frac{1}{2}\phi)\le \eta^2$, the hypothesis $\eta\le9$ implies $\zeta\le81$, and hence \cref{eq64,eq73} apply. Thus, under the hypotheses of the theorem, by \cref{eq71,eq72,eq73}
%%%%%%%%%%%%%%%%%
\begin{equation}
\label{eq75}
n^\alpha\frac{\partial}{\partial n}\mathcal 
F_n(\alpha,\beta,\zeta)
\ge
\alpha\beta(\alpha+\beta)\mathcal 
Q_5(\alpha,\beta,\zeta)>0.
\end{equation}
%%%%%%%%%%%%%%%%%
Thus, for fixed $\alpha,\beta,\zeta$, the function $\mathcal F_n(\alpha,\beta,\zeta)$ is strictly increasing as a function of $n$.

It remains to check the initial value $n=5$. For the parameter values satisfying the hypotheses of the theorem, by \cref{eq63,eq64}, we obtain
%%%%%%%%%%%%%%%%%
\begin{multline}
\label{eq76}
\mathcal F_5(\alpha,\beta,\zeta)
=5^{2-\beta}\left\{
H_5(\alpha,\beta,\zeta)
-0.0254\,\sigma(\alpha,\beta)\Gamma(2-\beta)4^{-2+\beta}
\right\}
\\
\ge
5^{2-\beta}\alpha\beta(\alpha+\beta)\,0.0388137394>0.
\end{multline}
%%%%%%%%%%%%%%%%%
Therefore $\mathcal F_n(\alpha,\beta,\zeta)>0$, and hence, by \cref{eq55,eq58,eq74},
%%%%%%%%%%%%%%%%%
\begin{equation}
\label{eq77}
w_n(\alpha,\beta,-1)-\bigl|
w_n(\alpha,\beta,e^{i\phi})\bigr|
\ge n^{-2+\beta}
\mathcal F_n(\alpha,\beta,\zeta)>0,
\end{equation}
%%%%%%%%%%%%%%%%%
which proves the theorem.
\end{proof}

\section{\texorpdfstring{Proof of \cref{thm:main-non-Meijer}}{Proof of Case II Theorem}}
\label{sec:case-II}

In this case $\phi>0$ since $\eta=n\phi>9$, and $\phi$ is small only when $n$ is large, since $n=\eta/\phi>9/\phi$. Moreover, because $\phi\le \phi_0=0.061$ throughout this section, the condition $\eta>9$ forces $n>9/\phi_0=9/0.061>147.5$, so this regime involves only large values of $n$. For this reason, we are able to modify the Watson approach of \cite{Dunster:2026:GBC}.

For convenience, we introduce the parameter
%%%%%%%%%%%%%%%%%
\begin{equation}
\label{eq78}
\nu=\alpha+\beta.
\end{equation}
%%%%%%%%%%%%%%%%%
Then define 
%%%%%%%%%%%%%%%%%
\begin{equation}
\label{eq79}
\widetilde{\mathsf{L}}(\alpha,\beta,\phi,s)
=
\frac{s^2}{\nu}
\left\{
\frac{\widetilde A_1(\alpha,\phi)\,
K_1(\alpha,\beta,\phi,s)}
{\Gamma(1-\beta)s^\beta}
-
\frac{s^\alpha \widetilde A_2(\alpha,\beta,\phi)\,
K_2(\alpha,\beta,\phi,s)}
{\Gamma(1+\alpha)}
\right\},
\end{equation}
%%%%%%%%%%%%%%%%%
where
%%%%%%%%%%%%%%%%%
\begin{equation}
\label{eq80}
\widetilde A_1(\alpha,\phi)
=
\frac{(1-\alpha)\bigl(2^\alpha-\{2\sin\left(\tfrac{1}{2}\phi\right)\}^\alpha\bigr)}
{\sin(\pi\alpha)\Gamma(1+\alpha)},
\end{equation}
%%%%%%%%%%%%%%%%%
%%%%%%%%%%%%%%%%%
\begin{equation}
\label{eq81}
\widetilde A_2(\alpha,\beta,\phi)
=
\frac{(1-\alpha)\Gamma(\beta)\bigl(\{2\sin\left(\tfrac{1}{2}\phi\right)\}^{-\beta}-2^{-\beta}\bigr)}{\pi},
\end{equation}
%%%%%%%%%%%%%%%%%
%%%%%%%%%%%%%%%%%
\begin{equation}
\label{eq82}
K_1(\alpha,\beta,\phi,s)
= \frac{1}{s^{2}}\left[\left(\frac{s}{e^s-1}\right)^{\beta}
\frac{(e^s+1)^\alpha - R(s,\phi)^\alpha}
{2^\alpha - \left\{2\sin(\tfrac12 \phi)\right\}^\alpha}- 1
- \frac{1}{2}(\alpha-\beta)\,s
\right],
\end{equation}
%%%%%%%%%%%%%%%%%
and
%%%%%%%%%%%%%%%%%
\begin{equation}
\label{eq83}
K_2(\alpha,\beta,\phi,s)
= \frac{1}{s^{2}}\left[\left(\frac{e^s-1}{s}\right)^{\alpha}
\frac{R(s,\phi)^{-\beta}-(e^s+1)^{-\beta}}
{\left\{2\sin(\tfrac12 \phi)\right\}^{-\beta}
-2^{-\beta} }- 1
- \frac{1}{2}(\alpha-\beta)\,s
\right].
\end{equation}
%%%%%%%%%%%%%%%%%

We shall work with the following Watson-type lower bound, derived in \cite{Dunster:2026:GBC} using \cref{eq08}, which does not require Meijer $G$-functions.

\begin{theorem}
\label{thm:Watson-lower-bound}

Let $\alpha,\beta \in (0,1)$ and $\phi \in (0,\pi)$. If $n\ge1$ is odd, then
%%%%%%%%%%%%%%%%%
\begin{multline}
\label{eq84}
w_n(\alpha,\beta,-1)-\bigl|w_n(\alpha,\beta,e^{i \phi})\bigr|
\ge
\frac{\nu\sin(\pi\alpha)\sin(\pi\beta)
\Gamma(1-\beta)\Gamma(1+\alpha)}
{(1-\alpha)\,n^{1-\beta}}
\\
\times
\left\{\frac{1}{\nu}
\widetilde A_1(\alpha,\phi)
\left(1+\frac{c_1(\alpha,\beta)}{n}\right)
-\frac{1}{\nu}\widetilde A_2(\alpha,\beta,\phi)n^{-\alpha-\beta}
\left(1+\frac{c_2(\alpha,\beta)}{n}\right)
\right.
\\
\left.
+n^{1-\beta}
\int_0^\infty
\widetilde{\mathsf{L}}(\alpha,\beta,\phi,s)e^{-ns}\,d s
\right\},
\end{multline}
%%%%%%%%%%%%%%%%%
where
%%%%%%%%%%%%%%%%%
\begin{equation}
\label{eq85}
c_1(\alpha,\beta)=\tfrac{1}{2}(\alpha-\beta)(1-\beta),
\quad
c_2(\alpha,\beta)=\tfrac{1}{2}(\alpha-\beta)(1+\alpha).
\end{equation}
%%%%%%%%%%%%%%%%%

\end{theorem}

\begin{proof}
The result follows from \cite[Thm.~3.2]{Dunster:2026:GBC} by substituting $\widetilde A_1(\alpha,\phi)=(\pi-\phi)^2 A_1(\alpha,\phi)$, $\widetilde A_2(\alpha,\beta,\phi)=(\pi-\phi)^2 A_2(\alpha,\beta,\phi)$, and $\widetilde{\mathsf L}(\alpha,\beta,\phi,s)=(\pi-\phi)^2 L(\alpha,\beta,\phi,s)$, and combining the two error integrals $\mathcal E_0(\alpha,\beta,\phi,n)$ and $\mathcal E_\infty(\alpha,\beta,\phi,n)$ of that theorem into a single integral over $(0,\infty)$.
\end{proof}

\begin{remark}
In the present setting the factor $(\pi-\phi)^{-2}$ used in \cite[Thm.~3.2]{Dunster:2026:GBC} is redundant, since here $\phi$ is bounded away from $\pi$, and has therefore been suppressed. Note also that $\widetilde A_2(\alpha,\beta,\phi)\to\infty$ as $\phi\to0$, but, as we shall show, this is compensated by the fact that $n\to\infty$ in the present case, since $n>9/\phi$.
\end{remark}

Our task is to show that the RHS of \cref{eq84} is positive according to the hypotheses of this section. The following is the first step, which shows that the remainder term in \cref{eq84} is not ``too negative''.

\begin{lemma}
\label{lem:Lraw-bound}
Let $\widetilde{\mathsf{L}}(\alpha,\beta,\phi,s)$ be defined by \cref{eq79}, and
%%%%%%%%%%%%%%%%%
\begin{equation}
\label{eq86}
\varpi(\alpha,\beta,\phi)
=
\begin{cases}
\tfrac{1}{3}\Gamma\left(\tfrac32-\beta\right)
\left(1-\tfrac{1}{9}\phi\right)^{-3/2+\beta}
& \quad \left((\alpha,\beta)\in S_0\right)\\[2mm]
\tfrac{1}{9}\Gamma(2-\beta)
\left(1-\tfrac{1}{9}\phi\right)^{-2+\beta}
& \quad \left((\alpha,\beta)\in S\right)
\end{cases},
\end{equation}
%%%%%%%%%%%%%%%%%
where
%%%%%%%%%%%%%%%%%
\begin{equation}
\label{eq87}
S_0=[0,1]\times\left[0,\tfrac14\right),
\quad
S=[0,1]\times\left[\tfrac14,1\right].
\end{equation}
%%%%%%%%%%%%%%%%%
Then, for $\alpha,\beta\in[0,1]$, $\phi \in (0,\phi_0]$, and $\eta=n\phi>9$,
%%%%%%%%%%%%%%%%%
\begin{equation}
\label{eq88}
n^{1-\beta}\int_0^\infty
\widetilde{\mathsf{L}}(\alpha,\beta,\phi,s)e^{-ns}\,d s
>-0.038 \, \varpi(\alpha,\beta,\phi).
\end{equation}
%%%%%%%%%%%%%%%%%
\end{lemma}
The proof is given in \cref{sec:lemma-proofs}.

Next, we have the following fundamental lower bound.

\begin{lemma}
\label{lem:non-Meijer}
For $\phi \in (0,\phi_0]$, $\alpha,\beta\in (0,1)$, and odd $n$ such that $\eta>9$,
%%%%%%%%%%%%%%%%%
\begin{equation}
\label{eq89}
w_n(\alpha,\beta,-1)
-\bigl|w_n(\alpha,\beta,e^{i \phi})\bigr|
>
\frac{\nu\sin(\pi\alpha)\sin(\pi\beta)
\Gamma(1-\beta)\Gamma(1+\alpha)}
{(1-\alpha)\,n^{1-\beta}}
\widetilde P(\alpha,\beta,\phi),
\end{equation}
%%%%%%%%%%%%%%%%%
where
%%%%%%%%%%%%%%%%%
\begin{multline}
\label{eq90}
\widetilde P(\alpha,\beta,\phi)
=\frac{1}{\nu}\widetilde A_1(\alpha,\phi)
\left(1-\frac{1-\beta}{18}[\beta-\alpha]_+\phi\right)
\\
-
\frac{1}{\nu}\widetilde A_2(\alpha,\beta,\phi)
\left(\frac{\phi}{9}\right)^{\nu}
\left(1+\frac{1+\alpha}{18}[\alpha-\beta]_+\phi\right)
-0.038 \, \varpi(\alpha,\beta,\phi).
\end{multline}
%%%%%%%%%%%%%%%%%

\end{lemma}

\begin{proof}
By \cref{lem:Lraw-bound}, the remainder term in \cref{eq84} is bounded below by the RHS of \cref{eq88}. It remains then only to bound the two explicit main terms in \cref{eq84}. For this use $\eta > 9$, together with
%%%%%%%%%%%%%%%%%
\begin{equation}
\label{eq91}
[-c_1(\alpha,\beta)]_+=\tfrac{1}{2}(1-\beta)[\beta-\alpha]_+,
\quad
[c_2(\alpha,\beta)]_+=\tfrac{1}{2}(1+\alpha)[\alpha-\beta]_+,
\end{equation}
%%%%%%%%%%%%%%%%%
to obtain
%%%%%%%%%%%%%%%%%
\begin{equation}
\label{eq92}
1+\frac{c_1(\alpha,\beta)\phi}{\eta}
\ge
1-\frac{1}{18}(1-\beta)[\beta-\alpha]_+\phi,
\end{equation}
%%%%%%%%%%%%%%%%%
and
%%%%%%%%%%%%%%%%%
\begin{equation}
\label{eq93}
\left(\frac{\phi}{\eta}\right)^{\nu}
\left(1+\frac{c_2(\alpha,\beta)\phi}{\eta}\right)
<
\left(\frac{\phi}{9}\right)^{\nu}
\left(1+\frac{1+\alpha}{18}[\alpha-\beta]_+\phi\right).
\end{equation}
%%%%%%%%%%%%%%%%%
Therefore, by \cref{eq84,eq92,eq93,eq88}, we get \cref{eq89}.
\end{proof}

The final piece used for our main theorem is the numerical result establishing positivity:

\begin{proposition}
For $\alpha,\beta\in[0,1]$ and $\phi \in [0,\phi_0]$ ($\phi_0=0.061$),
%%%%%%%%%%%%%%%%%
\begin{equation}
\label{eq94}
\widetilde P(\alpha,\beta,\phi)
\ge
\widetilde P(1,1,\phi_0)
=
\frac{1-\sin\left(\tfrac12 \phi_0\right)}{\pi}
-\frac{0.038}{9-\phi_0}
=0.3043519050\cdots.
\end{equation}
%%%%%%%%%%%%%%%%%
\end{proposition}

As an illustration, \cref{fig:tildeP} displays $\widetilde P(\alpha,\beta,\phi_0)$ on the full square $[0,1]^2$. The graph is consistent with the lower bound in \cref{eq94}, the smallest value being attained at the corner $(1,1)$.

\begin{figure}[hthp]
 \centering
 \includegraphics[
 width=0.9\textwidth,keepaspectratio]{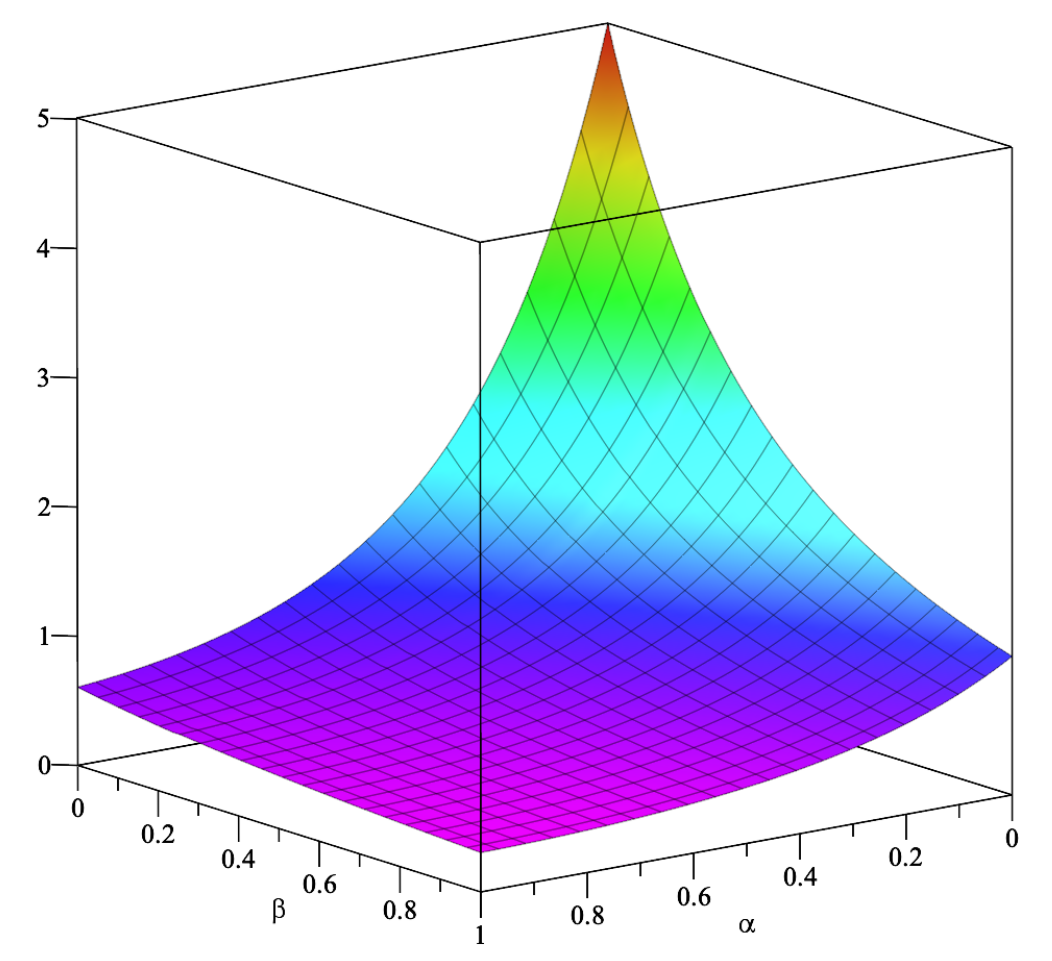}
 \caption{Graph of $\widetilde P(\alpha,\beta,\phi_0)$ for $\alpha,\beta \in [0,1]$.}
 \label{fig:tildeP}
\end{figure}

We now combine the preceding estimates to obtain the main result of this section, covering the non-Meijer range.

\begin{proof}[Proof of \cref{thm:main-non-Meijer}]
By \cref{lem:non-Meijer,eq94}, the lower bound in \cref{eq89}
is strictly positive for $\alpha,\beta\in(0,1)$, $\phi\in(0,\phi_0]$,
and $\eta>9$. The endpoint cases $\alpha=1$ and/or $\beta=1$ follow by continuous extension, using $\sin(\pi\alpha)/(1-\alpha)\to\pi$ as $\alpha\to1$ and $\sin(\pi\beta)\Gamma(1-\beta)\to\pi$ as $\beta\to1$.
\end{proof}

\section{\texorpdfstring{Proofs of \cref{lem:R-bound,lem:n-derivative,lem:Lraw-bound}}{Proofs of lemmas}} 
\label{sec:lemma-proofs}

Before proving \cref{lem:R-bound}, we record the endpoint behaviour of $\mathsf{L} (\alpha,\beta,\rho,s)$ that is used to justify the numerical lower bound \cref{eq106} below. Firstly, from \cref{eq57,eq09,eq10,eq29,eq32}, expanding at $s=0$ gives
%%%%%%%%%%%%%%%%%
\begin{multline}
\label{eq95}
\mathsf{L}(\alpha,\beta,\rho,s)
=
\tfrac{1}{24} \sin(\pi\beta) \left[
2^\alpha\left\{3(\alpha-\beta)^2+3\alpha-\beta\right\}
\right.
\\
\left.
-\rho^{\alpha}\left\{3(\alpha-\beta)^2-\beta\right\}
\right]s^{2-\beta}
+\mathcal{O}(s^{2+\alpha-\beta}),
\end{multline}
%%%%%%%%%%%%%%%%%
as $s \to 0$ for $\alpha,\beta \in [0,1]$ with $\nu=\alpha+\beta \ne 0$, and uniformly for $\rho \in [0,2\sin(\frac12 \phi_0)]$. 

The remaining limits follow from \cref{eq57,eq09,eq10,eq29,eq32,eq59} and by taking the corresponding boundary values.
%%%%%%%%%%%%%%%%%
\begin{multline}
\label{eq96}
\lim_{(\alpha,\beta)\to(0,0)}\frac{\mathsf{L}(\alpha,\beta,\rho,s)}{\sigma(\alpha,\beta)}
=
\frac1{8}\ln^2\left(\frac{s^2+\rho^2}{s^2}\right)
-\frac1{8}\ln^2\left(\frac{(e^s-1)^2+\rho^2\,e^s}{(e^s-1)^2}\right)
\\
+\frac12\ln^2\left(\frac{e^s+1}{e^s-1}\right)
-\frac12\ln^2\left(\frac{s}{2}\right)
=\mathcal{O}\left\{s^2\ln(s)\right\}
\quad (s \to 0),
\end{multline}
%%%%%%%%%%%%%%%%%
%%%%%%%%%%%%%%%%%
\begin{equation}
\label{eq97}
\lim_{\alpha\to0}\frac{\mathsf{L}(\alpha,\beta,\rho,s)}{\sigma(\alpha,\beta)}
=
\frac{1-\beta}{\beta\sin(\pi\beta)}
\left\{
\sin(\pi\beta)\,\partial_\alpha\mathcal A(0,\beta,\rho,s)
+\pi\,\mathcal B(0,\beta,\rho,s)
\right\},
\end{equation}
%%%%%%%%%%%%%%%%%
%%%%%%%%%%%%%%%%%
\begin{equation}
\label{eq98}
\lim_{\beta\to0}\frac{\mathsf{L}(\alpha,\beta,\rho,s)}{\sigma(\alpha,\beta)}
=
\frac{1-\alpha}{\alpha\sin(\pi\alpha)}
\left\{
\pi\,\mathcal A(\alpha,0,\rho,s)
+\sin(\pi\alpha)\,\partial_\beta\mathcal B(\alpha,0,\rho,s)
\right\},
\end{equation}
%%%%%%%%%%%%%%%%%%%%%%%%%%%%%%%%%%
where
%%%%%%%%%%%%%%%%%
\begin{equation}
\label{eq99}
\mathcal A(\alpha,\beta,\rho,s)
=
2^\alpha s^{2-\beta}\mathsf{K}_1(\alpha,\beta,s)
-
s^{2-\beta}(s^2+\rho^2)^{\alpha/2}\widetilde{\mathsf{K}}_1(\alpha,\beta,\rho,s),
\end{equation}
%%%%%%%%%%%%%%%%%
%%%%%%%%%%%%%%%%%
\begin{equation}
\label{eq100}
\mathcal B(\alpha,\beta,\rho,s)
=
2^{-\beta}s^{2+\alpha}\mathsf{K}_2(\alpha,\beta,s)
-
s^{2+\alpha}(s^2+\rho^2)^{-\beta/2}\widetilde{\mathsf{K}}_2(\alpha,\beta,\rho,s),
\end{equation}
%%%%%%%%%%%%%%%%%
%%%%%%%%%%%%%%%%%
\begin{equation}
\label{eq101}
\lim_{(\alpha,\beta)\to(0,1)}\frac{\mathsf{L}(\alpha,\beta,\rho,s)}{\sigma(\alpha,\beta)}
=
\frac{s}{4}-\frac{e^s-1}{2(e^s+1)}
-\frac{s-2}{2\sqrt{s^2+\rho^2}}
-\frac{1}{\sqrt{(e^s-1)^2+\rho^2\,e^s}},
\end{equation}
%%%%%%%%%%%%%%%%%
%%%%%%%%%%%%%%%%%
\begin{equation}
\label{eq102}
\lim_{(\alpha,\beta)\to(1,0)}\frac{\mathsf{L}(\alpha,\beta,\rho,s)}{\sigma(\alpha,\beta)}
=e^s-s-1
+\frac{1}{2}(s+2)\sqrt{s^2+\rho^2}
-\sqrt{(e^s-1)^2+\rho^2\,e^s},
\end{equation}
%%%%%%%%%%%%%%%%%
%%%%%%%%%%%%%%%%%
\begin{multline}
\label{eq103}
\lim_{\alpha\to1}
\frac{\mathsf{L}(\alpha,\beta,\rho,s)}{\sigma(\alpha,\beta)}
=
\frac{1-\beta}{2(1+\beta)}
\biggl[
2(e^s-1)^{-\beta}
\left(e^s+1-\sqrt{(e^s-1)^2+\rho^2\,e^s}\right)
\\
+s^{-\beta}\left(\sqrt{s^2+\rho^2}
-2\right)\left\{2+(1-\beta)s\right\}
\biggr],
\end{multline}
%%%%%%%%%%%%%%%%%
and
%%%%%%%%%%%%%%%%%
\begin{equation}
\label{eq104}
\lim_{\beta\to1}
\frac{\mathsf{L}(\alpha,\beta,\rho,s)}{\sigma(\alpha,\beta)}
=\frac{(1-\alpha)\mathcal{B}(\alpha,1,\rho,s)}{1+\alpha}.
\end{equation}
%%%%%%%%%%%%%%%%%

\begin{proof}[Proof of \cref{lem:R-bound}]
Let
%%%%%%%%%%%%%%%%%
\begin{equation}
\label{eq105}
\mathsf{L}_{0}
(\alpha,\beta,\rho,s)
=  
\frac{\mathsf{L}(\alpha,\beta,\rho,s)}
{\sigma(\alpha,\beta)s^{1-\beta} e^{s}}.
\end{equation}
%%%%%%%%%%%%%%%%%
We compute that
%%%%%%%%%%%%%%%%%
\begin{equation}
\label{eq106}
\inf \mathsf{L}_{0}(\alpha,\beta,\rho,s)
= -0.0253581508\cdots,
\end{equation}
%%%%%%%%%%%%%%%%%
where the infimum is taken over $\alpha,\beta \in [0,1]$, $\rho \in [0,2\sin(\tfrac12\phi_0)]$ and $s \in [0,\infty)$. Its existence follows from \cref{eq95,eq96,eq97,eq98,eq99,eq100,eq101,eq102,eq103,eq104}, and we find it is attained at $(\alpha,\beta,\rho,s)=(0,1,0,s_0)$ where $s_0 = 0.7149617952\cdots$.

\begin{figure}[hthp]
 \centering
 \includegraphics[
 width=0.9\textwidth,keepaspectratio]{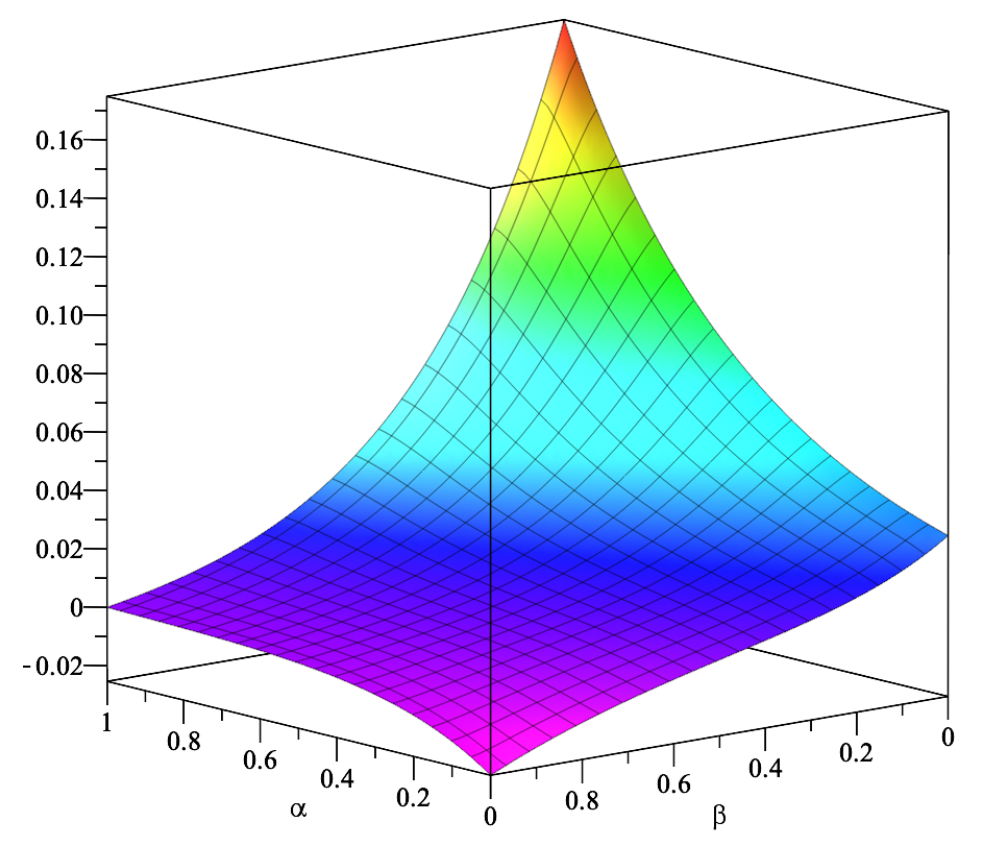}
 \caption{Graph of $\mathsf{L}_{0}(\alpha,\beta,0,s_0)$ for $\alpha,\beta \in [0,1]$.}
 \label{fig:Meijer-m0}
\end{figure}

Hence in this box
%%%%%%%%%%%%%%%%%
\begin{equation}
\label{eq107}
\mathsf{L}(\alpha,\beta,\rho,s) > -0.0254\,\sigma(\alpha,\beta)
s^{1-\beta} e^{s},
\end{equation}
%%%%%%%%%%%%%%%%%
and thus we get
%%%%%%%%%%%%%%%%%
\begin{equation}
\label{eq108}
\int_0^\infty \mathsf{L}(\alpha,\beta,\rho,s)
e^{-ns}\,d s
>
-0.0254\,\sigma(\alpha,\beta)\int_0^\infty 
s^{1-\beta}e^{-(n-1)s} ds,
\end{equation}
%%%%%%%%%%%%%%%%%
with \cref{eq58} then following from \cref{eq14}.
\end{proof}

The plot in \cref{fig:Meijer-m0} gives a two-dimensional slice of the normalised remainder kernel at the numerically extremal values $\rho=0$ and $s=s_0$. It illustrates the boundary minimum at $(0,1)$ as given in \cref{eq106}.

\begin{proof}[Proof of \cref{lem:n-derivative}]

From \cref{eq56,eq74}, for $n>1$, we have by explicit differentiation
%%%%%%%%%%%%%%%%%
\begin{multline}
\label{eq109}
n^{\alpha}\frac{\partial}
{\partial n}\left\{ n^{2-\beta}
\left[ H_n(\alpha,\beta,\zeta)
-0.0254\,\sigma(\alpha,\beta)
\Gamma(2-\beta)\,(n-1)^{-2+\beta} \right]\right\}
\\
= \hat{H}_n(\alpha,\beta,\zeta)
+0.0254\,\sigma(\alpha,\beta)\,\Gamma(3-\beta)
\,n^{1+\alpha-\beta}(n-1)^{-3+\beta}
\ge \hat{H}_n(\alpha,\beta,\zeta),
\end{multline}
%%%%%%%%%%%%%%%%%
where
%%%%%%%%%%%%%%%%%
\begin{equation}
\label{eq110}
\hat{H}_n(\alpha,\beta,\zeta)
=
n^{\alpha}\frac{\partial}{\partial n}
\left\{n^{2-\beta}
H_n(\alpha,\beta,\zeta)\right\}.
\end{equation}
%%%%%%%%%%%%%%%%%
Treating $\zeta$ as independent of $n$, we obtain from \cref{eq56,eq110}
%%%%%%%%%%%%%%%%%
\begin{multline}
\label{eq111}
\hat H_n(\alpha,\beta,\zeta)
=
h_1(\alpha,\beta)n^\alpha
+h_2(\alpha,\beta)n^{-\beta}
+h_3(\alpha,\beta,\zeta)
+n^{-1}g_n(\alpha,\beta,\zeta),
\end{multline}
%%%%%%%%%%%%%%%%%
where $h_1(\alpha,\beta)$, $h_2(\alpha,\beta)$, $h_3(\alpha,\beta,\zeta)$ are given by \cref{eq67,eq68,eq69}, and
%%%%%%%%%%%%%%%%%
\begin{multline}
\label{eq112}
g_n(\alpha,\beta,\zeta)
=
\frac{1}{2}\alpha(\alpha-\beta)\sin(\pi\beta)
M_{1,1}(\alpha,\beta,\zeta)
\\
+\frac{1}{2}\sin(\pi\alpha)\Gamma(2+\alpha)
(\alpha-\beta)\left\{\alpha
\frac{M_{2,1}(\alpha,\beta,\zeta)}
{\Gamma(2+\alpha)}
-\frac{\alpha+\beta}{(2n)^{\beta}}
\right\}.
\end{multline}
%%%%%%%%%%%%%%%%%

Now consider bounding from below the first two terms combined on the RHS of \cref{eq111}, and as such define
%%%%%%%%%%%%%%%%%
\begin{equation}
\label{eq113}
d_n(\alpha,\beta)=
h_1(\alpha,\beta)n^\alpha+h_2(\alpha,\beta)n^{-\beta}.
\end{equation}
%%%%%%%%%%%%%%%%%
If $h_2(\alpha,\beta)>0$ the simple bound $d_n(\alpha,\beta) \ge h_1(\alpha,\beta)5^\alpha$ for $n \ge 5$ is not sufficiently sharp for our purposes. Instead, consider the derivative
%%%%%%%%%%%%%%%%%
\begin{equation}
\label{eq114}
\frac{\partial}{\partial n}d_n(\alpha,\beta)
=
\alpha h_1(\alpha,\beta)n^{\alpha-1}
-\beta h_2(\alpha,\beta)n^{-\beta-1}.
\end{equation}
%%%%%%%%%%%%%%%%%
From this we find that $d_n(\alpha,\beta)$, regarded as a function of $n$, has the unique critical point
%%%%%%%%%%%%%%%%%
\begin{equation}
\label{eq115}
n=n^*:=
\left(\frac{\beta h_2(\alpha,\beta)}{\alpha h_1(\alpha,\beta)}\right)^{1/(\alpha+\beta)}.
\end{equation}
%%%%%%%%%%%%%%%%%
Moreover, from \cref{eq67,eq68,eq78} we have
%%%%%%%%%%%%%%%%%
\begin{equation}
\label{eq116}
\frac{\beta h_2(\alpha,\beta)}{\alpha 
h_1(\alpha,\beta)}
=
\frac{(1-\nu)\Gamma(1+\beta)}{2^{\nu}
\,\Gamma(1-\alpha)},
\end{equation}
%%%%%%%%%%%%%%%%%
where $\nu=\alpha+\beta \in (0,1)$ (on the assumption $h_2(\alpha,\beta)>0$; cf. \cref{eq68,eq78}). Now, since $2^{\nu}>1$, $0<1-\nu<1$, $\Gamma(1+\beta)\le 1$, and $\Gamma(1-\alpha)\ge 1$ for $\alpha,\beta\in[0,1)$, the RHS of \cref{eq116} lies strictly between $0$ and $1$. Hence from \cref{eq115} $n^* \in (0,1)$, and taking into account $d_n(\alpha,\beta)\to+\infty$ as $n\to\infty$, we conclude that the minimum of $d_n(\alpha,\beta)$ on $n\ge5$ is attained at $n=5$ when $h_2(\alpha,\beta)>0$.

On the other hand, if $h_2(\alpha,\beta)\le 0$, then both terms on the RHS of \cref{eq114} are nonnegative, so $d_n(\alpha,\beta)$ is nondecreasing for $n>0$, and again its minimum on $n\ge 5$ is attained at $n=5$. Thus, in both cases,
%%%%%%%%%%%%%%%%%
\begin{equation}
\label{eq117}
d_n(\alpha,\beta)\ge d_5(\alpha,\beta)
=
h_1(\alpha,\beta)5^\alpha+h_2(\alpha,\beta)5^{-\beta}
\quad (n\ge 5).
\end{equation}
%%%%%%%%%%%%%%%%%

It remains to bound the correction term $g_n(\alpha,\beta,\zeta)$ in \cref{eq111}. Now from \cref{eq50}
%%%%%%%%%%%%%%%%%
\begin{equation}
\label{eq118}
\alpha\,\frac{M_{2,1}(\alpha,\beta,\zeta)}
{\Gamma(2+\alpha)}
-\frac{\alpha+\beta}{(2n)^{\beta}}
=\mathcal{O}(\beta)
\quad (\beta \to 0).
\end{equation}
%%%%%%%%%%%%%%%%%
Also the first term on the RHS of \cref{eq112} is obviously $\mathcal{O}(\beta)$ as $\beta\to0$, and in our bounds we wish to maintain this overall behaviour of the RHS of \cref{eq112}. So with this in mind, add and subtract $\alpha$ inside the bracket, yielding
%%%%%%%%%%%%%%%%%
\begin{multline}
\label{eq119}
g_n(\alpha,\beta,\zeta)
=
\frac{1}{2}\alpha(\alpha-\beta)\sin(\pi\beta)
M_{1,1}(\alpha,\beta,\zeta)
\\
+\frac{1}{2}\sin(\pi\alpha)\Gamma(2+\alpha)
\Biggl\{
\alpha(\alpha-\beta)
\left(\frac{M_{2,1}(\alpha,\beta,\zeta)}
{\Gamma(2+\alpha)}-1\right)
\\
+(\alpha-\beta)
\left(\alpha-
\frac{\alpha+\beta}
{(2n)^{\beta}}\right)
\Biggr\}.
\end{multline}
%%%%%%%%%%%%%%%%%
From \cref{eq50} note that
%%%%%%%%%%%%%%%%%
\begin{equation}
\label{eq120}
\frac{M_{2,1}(\alpha,\beta,\zeta)}
{\Gamma(2+\alpha)}-1
=\mathcal{O}(\beta)
\quad (\beta \to 0).
\end{equation}
%%%%%%%%%%%%%%%%%

For the elementary part, let
%%%%%%%%%%%%%%%%%
\begin{equation}
\label{eq121}
r_n(\alpha,\beta)=(\alpha-\beta)
\{\alpha-(\alpha+\beta)
(2n)^{-\beta}\},
\end{equation}
%%%%%%%%%%%%%%%%%
which also is $\mathcal{O}(\beta)$ as $\beta \to 0$. Now, for $n\ge 5$ and $\alpha\ge\beta$,
%%%%%%%%%%%%%%%%%
\begin{equation}
\label{eq122}
r_n(\alpha,\beta)
\ge
\alpha(\alpha-\beta)-(\alpha^2-\beta^2)10^{-\beta}.
\end{equation}
%%%%%%%%%%%%%%%%%
If $\alpha\le\beta$, then $[\alpha^2-\beta^2]_+=0$ and $r_n(\alpha,\beta)\ge \alpha(\alpha-\beta)$.
Therefore, in both cases,
%%%%%%%%%%%%%%%%%
\begin{equation}
\label{eq123}
r_n(\alpha,\beta)\ge
\alpha(\alpha-\beta)-[\alpha^2-\beta^2]_+\,10^{-\beta}
\quad (n\ge 5).
\end{equation}
%%%%%%%%%%%%%%%%%
This lower bound is $\mathcal{O}(\beta)$ as $\beta \to 0$, as desired. Consequently, by \cref{eq65,eq119,eq123}, for $n\ge 5$,
%%%%%%%%%%%%%%%%%
\begin{equation}
\label{eq124}
n^{-1}g_n(\alpha,\beta,\zeta)
\ge -\tfrac12 n^{-1}h_4(\alpha,\beta,\zeta)
\ge -\tfrac{1}{10}[h_4(\alpha,\beta,\zeta)]_+,
\end{equation}
%%%%%%%%%%%%%%%%%
where $h_4(\alpha,\beta,\zeta)$ is given by \cref{eq70}.

Combining \cref{eq111,eq117,eq124}, we obtain
%%%%%%%%%%%%%%%%%
\begin{equation}
\label{eq125}
\hat H_n(\alpha,\beta,\zeta)
\ge
h_1(\alpha,\beta)5^\alpha+h_2(\alpha,\beta)5^{-\beta}
+h_3(\alpha,\beta,\zeta)
-\tfrac{1}{10}[h_4(\alpha,\beta,\zeta)]_+,
\end{equation}
%%%%%%%%%%%%%%%%%
valid for $n\ge 5$. The lower bound \cref{eq71} then follows from \cref{eq109,eq125}.
\end{proof}

We finally turn our attention to the proof of \cref{lem:Lraw-bound}. Before doing so, we first record some elementary limits and asymptotic formulae for
$\widetilde{\mathsf{L}}(\alpha,\beta,\phi,s)$ (defined by \cref{eq79}), which describe its endpoint behaviour and motivate the normalisation used below. 

In \cref{eq126,eq127}, recall that $\rho=2\sin(\frac{1}{2}\phi)$. Now, as $s\to 0$, with $\nu=\alpha+\beta>0$ and $\rho>0$,
%%%%%%%%%%%%%%%%%
\begin{multline}
\label{eq126}
K_1(\alpha,\beta,\phi,s)
=\frac{1}{24\rho^2\left(2^\alpha-\rho^\alpha\right)}
\left[
\left\{\beta-3(\alpha-\beta)^2\right\}\rho^{2+\alpha}
\right.
\\
\left.
+\left\{3(\alpha-\beta)^2+3\alpha-\beta\right\}2^\alpha\rho^2
-12\alpha\rho^\alpha\right]+\mathcal{O}(s),
\end{multline}
%%%%%%%%%%%%%%%%%
and
%%%%%%%%%%%%%%%%%
\begin{multline}
\label{eq127}
K_2(\alpha,\beta,\phi,s)
=\frac{1}{24\rho^2\left(2^\beta-\rho^\beta\right)}
\left[
\left\{3\beta-\alpha-3(\alpha-\beta)^2\right\}\rho^{2+\beta}
\right.
\\
\left.
+\left\{3(\alpha-\beta)^2+\alpha\right\}2^\beta\rho^2
-12\beta\,2^\beta
\right]+\mathcal{O}(s).
\end{multline}
%%%%%%%%%%%%%%%%%
The limiting cases $\alpha=0$ and $\beta=0$ are understood by continuous extension and are recorded separately below. From \cref{eq79,eq126,eq127} it follows that 
%%%%%%%%%%%%%%%%%
\begin{equation}
\label{eq128}
\widetilde{\mathsf{L}}(\alpha,\beta,\phi,s)
=
\mathcal{O}(s^{2-\beta})
\quad (\nu>0,\, \phi>0, \,s \to 0).
\end{equation}
%%%%%%%%%%%%%%%%%
Moreover
%%%%%%%%%%%%%%%%%
\begin{multline}
\label{eq129}
\widetilde{\mathsf{L}}(0,\beta,\phi,s)
=
\frac{\Gamma(\beta)}{\pi\beta}
\biggl[
(e^{s}+1)^{-\beta}-R(s,\phi)^{-\beta}
\biggr.
\\
\biggl.
+\left(\tfrac{1}{2}\beta s-1\right)
\left(2^{-\beta}
-\left\{2\sin(\tfrac12\phi)\right\}
^{-\beta}\right)
\biggr]
\\
+\frac{1}{\pi\,\Gamma(1-\beta)\,\beta}
\left[
(e^{s}-1)^{-\beta}
\ln\left(\frac{e^{s}+1}{R(s,\phi)}\right)
+\left(\tfrac{1}{2}\beta s^{1-\beta}
-s^{-\beta}\right)
\ln\left\{\csc(\tfrac12\phi)\right\}
\right],
\end{multline}
%%%%%%%%%%%%%%%%%
and
%%%%%%%%%%%%%%%%%
\begin{multline}
\label{eq130}
\widetilde{\mathsf{L}}(\alpha,0,\phi,s)
=
\frac{1-\alpha}{\alpha\,\Gamma(1+\alpha)}
\\ \times
\left[
\frac{1}{\sin(\pi\alpha)}
\left\{
(e^{s}+1)^{\alpha}-R(s,\phi)^{\alpha}
-\left(1+\frac{1}{2}\alpha s\right)
\left(2^{\alpha}
-\left\{2\sin(\tfrac12\phi)\right\}^{\alpha}\right)
\right\}
\right.
\\
\left.
-\frac{s^{\alpha}}{\pi}
\left\{
\left(\frac{e^{s}-1}{s}\right)^{\alpha}
\ln\left(\frac{e^{s}+1}{R(s,\phi)}\right)
-\left(1+\frac{1}{2}\alpha s\right)
\ln\left\{\csc(\tfrac12\phi)\right\}
\right\}
\right].
\end{multline}
%%%%%%%%%%%%%%%%%
Also for $\phi>0$
%%%%%%%%%%%%%%%%%
\begin{multline}
\label{eq131}
\widetilde{\mathsf{L}}(0,0,\phi,s)
=\frac{1}{2\pi}
\Biggl[2\ln\left\{\frac{R(s,\phi)}{e^{s}+1}\right\}
\ln\left(e^{s}-1\right)
+\ln^{2}\left(e^{s}+1\right)
\biggr. \\ \left.
-\ln^{2}\left\{R(s,\phi)\right\}
+\ln\left\{\sin(\tfrac{1}{2}\phi)\right\}
\ln\left\{\frac{4\sin(\tfrac{1}{2}\phi)}
{s^{2}}\right\}
\right]
\\
=\frac{\left\{1+\cos(\phi)\right\}s^{2}\ln(s)}
{8\pi\left\{1-\cos(\phi)\right\}}
+\mathcal{O}\left(s^{2}\right)
\quad (s \to 0).
\end{multline}
%%%%%%%%%%%%%%%%%
Furthermore $\widetilde{\mathsf{L}}(\alpha,\beta,\phi,s)$ does not vanish at $s=0$ in the following special case
%%%%%%%%%%%%%%%%%
\begin{equation}
\label{eq132}
\widetilde{\mathsf{L}}(0,0,\phi,c \phi)
\to
\frac{1}{8\pi}\ln(1+c^2)\,
\ln\left(\frac{c^4}{1+c^2}\right)
\quad (\phi \to 0, \, c>0).
\end{equation}
%%%%%%%%%%%%%%%%%
We remark that this attains a minimum value $-0.02911\cdots$ at $c=0.66036\cdots$.

As $\phi \to 0$ 
%%%%%%%%%%%%%%%%%
\begin{equation}
\label{eq133}
\widetilde{\mathsf{L}}(\alpha,\beta,\phi,s)
=
\mathcal{O}\left(\phi^{-\beta}\right),
\end{equation}
%%%%%%%%%%%%%%%%%
for fixed $s>0$, $\alpha\in [0,1]$ and $\beta \in (0,1]$. On the edge $\beta=0$, \cref{eq130} instead gives
%%%%%%%%%%%%%%%%%
\begin{equation}
\label{eq134}
\widetilde{\mathsf L}(\alpha,0,\phi,s)
=\mathcal O\left(\ln(1/\phi)\right)
\quad (\phi\to0),
\end{equation}
%%%%%%%%%%%%%%%%%
for fixed $s>0$ and $\alpha\in(0,1]$. Thus it is generally unbounded in this limit. 

Finally, as $s \to \infty$,
%%%%%%%%%%%%%%%%%
\begin{equation}
\label{eq135}
\widetilde{\mathsf{L}}(\alpha,\beta,\phi,s)
=\frac{\alpha-\beta}{2(\alpha+\beta)}\left\{
\frac{\widetilde A_2(\alpha,\beta,\phi)}{\Gamma(1+\alpha)}\,s^{1+\alpha}
-\frac{\widetilde A_1(\alpha,\phi)}
{\Gamma(1-\beta)}s^{1-\beta}
\right\}+\mathcal{O}(s^{\alpha}).
\end{equation}

With the above limits in mind, we set
%%%%%%%%%%%%%%%%%
\begin{equation}
\label{eq136}
\widetilde{\mathsf{L}}_{0}(\alpha,\beta,\phi,s)
=
\begin{cases}
s^{-1/2+\beta}\,e^{-s}\phi^{1/2}\,
\widetilde{\mathsf{L}}
\left(\alpha,\beta,\phi,s\right)
& \quad \left((\alpha,\beta)
\in S_0\right)
\\[2mm]
s^{-1+\beta}\,e^{-s}\phi\,
\widetilde{\mathsf{L}}
\left(\alpha,\beta,\phi,s\right)
& \quad \left((\alpha,\beta)
\in S\right)
\end{cases},
\end{equation}
%%%%%%%%%%%%%%%%%
where $S$ and $S_0$ are given by \cref{eq87}. These two different normalisations will be explained after the following proof; see \cref{rem:tildeL0}.

We note that on $S$ the limit of $\widetilde{\mathsf L}_0(\alpha,\beta,\phi,s)$ as $\phi\to0$ is zero when $\frac14 \le \beta<1$. On the edge $\beta=1$, however, the limiting value is generally non-zero, specifically
%%%%%%%%%%%%%%%%%
\begin{equation}
\label{eq137}
\widetilde{\mathsf L}_0(\alpha,1,0;s)
=
\frac{(1-\alpha)e^{-s}s^\alpha}
{\pi\Gamma(2+\alpha)}
\left\{1-\frac{1}{2}(1-\alpha)s\right\}.
\end{equation}
%%%%%%%%%%%%%%%%%
For $\alpha\in[0,1]$ and $s\in[0,\infty)$ this has minimum value $-1/(2\pi e^3)=-0.007924\cdots$, attained at $\alpha=0$, $s=3$. Maple's \texttt{Optimization[Minimize]} routine over $\alpha\in[0,1]$, $s\in[0,10]$ returns this point. For the remaining tail $s\ge10$, we use the cruder lower bound $s(1-\frac12 s)/(2\pi e^s)$. On setting $t=1/s$, Maple's \texttt{Optimization[Minimize]} routine on $0\le t\le0.1$ gives the minimum at $t=0.1$ (that is, $s=10$), with value $-20/(\pi e^{10})=-0.0002890249\cdots$.

We also remark that $s^\alpha$ has no unique limit as $(\alpha,s)\to(0^+,0^+)$. However, the expression in \cref{eq137} is harmless at this point, since $s^\alpha=\exp\{\alpha\ln(s)\}$, and therefore $\lim \inf_{(\alpha,s)\to (0^+,0^+)}\widetilde{\mathsf L}_0(\alpha,1,0;s) =0$. Thus the corner $\alpha=s=0$ does not affect the negative minimum.

\begin{proof}[Proof of \cref{lem:Lraw-bound}]
For $\alpha,\beta \in [0,1]$, $\phi\in[0,\phi_0]$ and $s\in(0,\infty)$, we numerically evaluate
%%%%%%%%%%%%%%%%%
\begin{equation}
\label{eq138}
\widetilde{m}_0
=
\inf\widetilde{\mathsf{L}}_{0}(\alpha,\beta,\phi,s)
=-0.0377104213\cdots.
\end{equation}
%%%%%%%%%%%%%%%%%
The existence of this infimum follows from the endpoint behaviour described in \cref{eq126,eq127,eq128,eq129,eq130,eq131,eq132,eq133,eq135}, together with the definition \cref{eq136}. The infimum is approached in the corner $\alpha=\beta=0$, in the limiting case $s=c_0\phi$, $\phi\to0$, with $c_0=0.5281398736\cdots$; see \cref{sec:numerics}, where this minimum is found by a one-variable calculus calculation.

In \cref{fig:tildem0} a graph of $\widetilde{\mathsf{L}}_{0}(\alpha,\beta,\phi,s)$ is given for $\alpha,\beta\in [0,1]$, where $\phi=10^{-50}$ and $s=0.52814 \times 10^{-50}$, illustrating the limiting case $\alpha=\beta=0$, $s=c_0\phi$, and $\phi \to 0$, where the infimum in \cref{eq138} is approached.

\begin{figure}[hthp]
 \centering
 \includegraphics[
 width=0.9\textwidth,keepaspectratio]{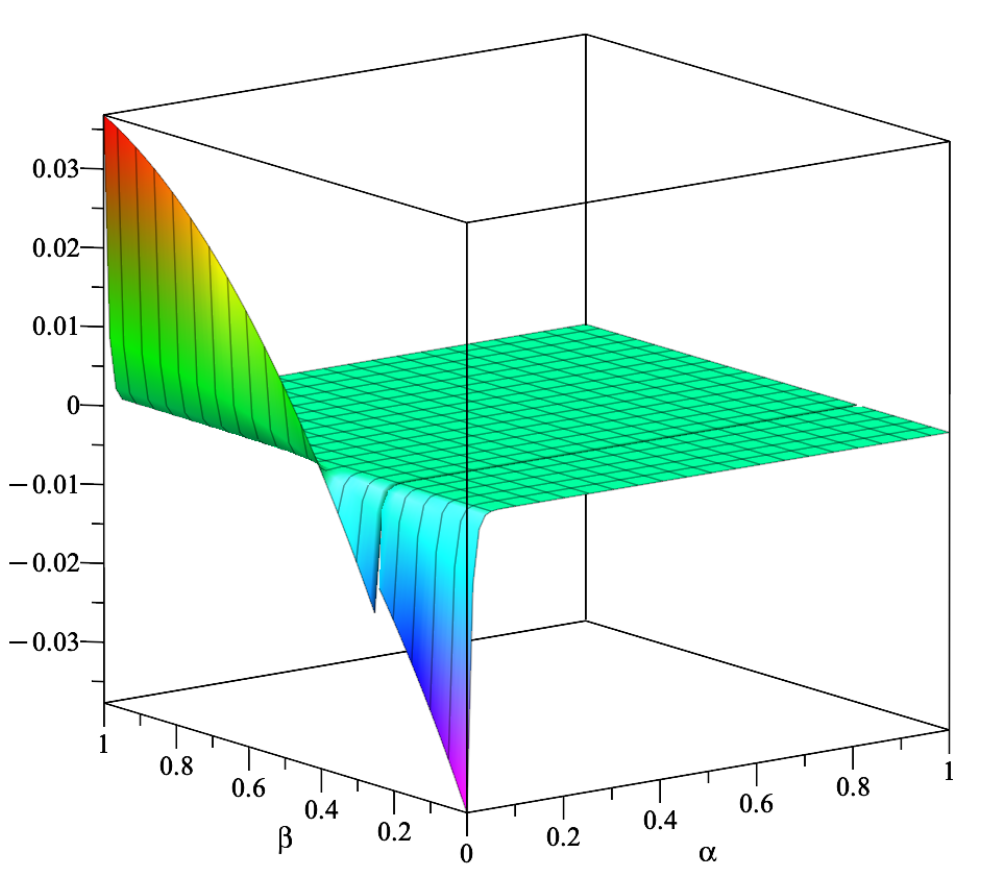}
 \caption{Graph of $\widetilde{\mathsf{L}}_{0}(\alpha,\beta,\phi,s)$ for $\alpha,\beta \in [0,1]$, $\phi=10^{-50}$ and $s=0.52814 \times 10^{-50}$.}
 \label{fig:tildem0}
\end{figure}

Next, suppose $(\alpha,\beta)\in S_0$; then by \cref{eq136,eq138},
%%%%%%%%%%%%%%%%%
\begin{equation}
\label{eq139}
\widetilde{\mathsf{L}}(\alpha,\beta,\phi,s)
\ge
\widetilde m_0\,e^{s}\phi^{-1/2}s^{1/2-\beta},
\end{equation}
%%%%%%%%%%%%%%%%%
and hence, on referring to \cref{eq07,eq14}, we have in the rectangle $S_0$
%%%%%%%%%%%%%%%%%
\begin{multline}
\label{eq140}
n^{1-\beta}\int_0^\infty
\widetilde{\mathsf{L}}(\alpha,\beta,\phi,s)e^{-ns}\,d s
\ge
\widetilde m_0\,n^{1-\beta}\phi^{-1/2}
\int_0^\infty s^{1/2-\beta}e^{-(n-1)s}\,d s
\\
=\widetilde m_0\,\phi^{-1/2}\Gamma
\left(\tfrac32-\beta\right)n^{1-\beta}(n-1)^{-3/2+\beta}
=\widetilde m_0\,\Gamma\left(\tfrac32-\beta\right)
\eta^{-1/2}\left(1-\phi \eta^{-1}\right)^{-3/2+\beta}.
\end{multline}
%%%%%%%%%%%%%%%%%
Since $\eta^{-1/2}\left(1- \eta^{-1}\phi\right)^{-3/2+\beta}$ is decreasing for $\eta>\phi$, and $\widetilde m_0<0$, we conclude from \cref{eq140} and $\eta>9$ that
%%%%%%%%%%%%%%%%%
\begin{equation}
\label{eq141}
n^{1-\beta}\int_0^\infty
\widetilde{\mathsf{L}}(\alpha,\beta,\phi,s)
e^{-ns}\,d s
>
\frac{\widetilde m_0}{3}\Gamma
\left(\frac32-\beta\right)
\left(1-\frac{\phi}{9}\right)^{-3/2+\beta},
\end{equation}
%%%%%%%%%%%%%%%%%
for $(\alpha,\beta) \in S_0$.

Now consider $(\alpha,\beta)\in S$. From \cref{eq136,eq138} we then have
%%%%%%%%%%%%%%%%%
\begin{equation}
\label{eq142}
\widetilde{\mathsf{L}}(\alpha,\beta,\phi,s)
\ge
\widetilde m_0\,e^{s}
\phi^{-1}s^{1-\beta},
\end{equation}
%%%%%%%%%%%%%%%%%
and therefore, again from \cref{eq07,eq14},
%%%%%%%%%%%%%%%%%
\begin{multline}
\label{eq143}
n^{1-\beta}\int_0^\infty
\widetilde{\mathsf{L}}(\alpha,\beta,\phi,s)e^{-ns}\,d s
\ge
\widetilde m_0\,n^{1-\beta}
\phi^{-1}\int_0^\infty s^{1-\beta}e^{-(n-1)s}\,d s
\\
= \widetilde m_0\,\phi^{-1}
\Gamma(2-\beta)n^{1-\beta}(n-1)^{-2+\beta}
=\widetilde m_0\,\Gamma(2-\beta)
\eta^{-1}
\left(1-\phi \eta^{-1} \right)^{-2+\beta},
\end{multline}
%%%%%%%%%%%%%%%%%
in this rectangle. Since $\eta^{-1}(1- \eta^{-1} \phi)^{-2+\beta}$ is decreasing for $\eta>\phi$, and $\widetilde m_0<0$, it follows from \cref{eq143} and $\eta>9$ that
%%%%%%%%%%%%%%%%%
\begin{equation}
\label{eq144}
n^{1-\beta}\int_0^\infty
\widetilde{\mathsf{L}}(\alpha,\beta,\phi,s)e^{-ns}\,d s
>
\frac{\widetilde m_0}{9}\,\Gamma(2-\beta)
\left(1-\frac{\phi}{9}\right)^{-2+\beta},
\end{equation}
%%%%%%%%%%%%%%%%%
for $(\alpha,\beta)\in S$. Combining \cref{eq141,eq144}, and using from \cref{eq138} that $\widetilde m_0>-0.038$, yields \cref{eq88} where $\varpi(\alpha,\beta,\phi)$ is given by \cref{eq86}.
\end{proof}

\begin{remark}
\label{rem:tildeL0}
The reason why $\widetilde{\mathsf{L}}_{0}(\alpha,\beta,\phi,s)$ is defined in two pieces in \cref{eq136} is to keep both $\widetilde{\mathsf{L}}_{0}(\alpha,\beta,\phi,s)$ and $\varpi(\alpha,\beta,\phi)$ bounded in the relevant endpoint regimes. The rectangle $S_0$ contains the edge $\beta=0$ and the joint corner $\alpha=\beta=0$, where the endpoint $s,\phi\to0$ is most delicate. In this region the factor $\phi^{1/2}$ yields \cref{eq141}, which does not contain a negative power of $\phi$ and therefore remains bounded as $\phi\to0$.

The power of $s$ has a different role. In $S_0$ the worst behaviour occurs when $s$ and $\phi$ tend to zero together. In particular, the limiting case $\alpha=\beta=0$, $s=c\phi$, shows that the more natural normalisation with $s^{-1+\beta}$ is too singular as $\phi\to0$; see \cref{eq132}. Using the weaker factor $s^{-1/2+\beta}$ controls this corner and still leads to the explicit integral estimate \cref{eq140}.

On the complementary rectangle $S$, the parameter $\beta$ is bounded away from zero but may be as large as $1$. For this reason the factor $\phi$ is needed in this region, specifically, for fixed $s>0$, one generally has the singular behaviour \cref{eq133}. Thus $\phi\widetilde{\mathsf L}(\alpha,\beta,\phi,s)$ remains bounded on $S$, whereas $\phi^{1/2}\widetilde{\mathsf L}(\alpha,\beta,\phi,s)$ would not remain bounded as $\phi\to 0$ when $\beta>1/2$. The accompanying power $s^{-1+\beta}$ is chosen to obtain the uniform integral estimate \cref{eq144}. If $s^{-1/2+\beta}$ were used instead on $S$, the corresponding integrated bound \cref{eq144} would contain an extra factor $\phi^{-1/2}$ and therefore again would not be uniform as $\phi\to0$.

The cutoff value $\frac14$ in the definitions \cref{eq87} is not intrinsic; it is simply a convenient fixed value in $(0,\frac12)$ separating $S_0$ which contains small values of $\beta$ from $S$ that is bounded away from the edge $\beta=0$.
\end{remark}

\section{Numerical computations} 
\label{sec:numerics}

For the numerical verifications leading to \cref{eq64,eq73}, the dependence on $\zeta\in[0,81]$ was first eliminated analytically using the following two lemmas. The resulting lower bounds require numerical Meijer-function evaluations only at the fixed endpoint $\zeta=81$; the values arising at $\zeta=0$ reduce explicitly to gamma functions by \cref{eq46}.

\begin{lemma}
\label{lem:alpha<beta}
Let
%%%%%%%%%%%%%%%%%
\begin{equation}
\label{eq145}
\Psi_{0}(\alpha,\beta,\zeta)
=
\sin(\pi\beta)M_{1,0}(\alpha,\beta,\zeta)
+\sin(\pi\alpha)M_{2,0}(\alpha,\beta,\zeta),
\end{equation}
%%%%%%%%%%%%%%%%%
and
%%%%%%%%%%%%%%%%%
\begin{equation}
\label{eq146}
\Psi_{1}(\alpha,\beta,\zeta)
=
\sin(\pi\beta)M_{1,1}(\alpha,\beta,\zeta)
+\sin(\pi\alpha)M_{2,1}(\alpha,\beta,\zeta).
\end{equation}
%%%%%%%%%%%%%%%%%
For all $\alpha,\beta\in(0,1)$, the function
$\Psi_{0}(\alpha,\beta,\zeta)$ is monotonically increasing with respect to $\zeta$ on $[0,\infty)$. If $0<\beta\le\alpha<1$, then $\Psi_{1}(\alpha,\beta,\zeta)$ is also monotonically increasing with respect to $\zeta$ on $[0,\infty)$.
\end{lemma}

\begin{proof}
By \cref{eq40,eq42},
%%%%%%%%%%%%%%%%%
\begin{multline}
\label{eq147}
\frac{\partial \Psi_{0}(\alpha,\beta,\zeta)}
{\partial \zeta}
=
\frac12\int_{0}^{\infty}t^{-\beta}(t^{2}
+\zeta)^{-\beta/2-1}
\\ \times
\left[
\alpha\sin(\pi\beta)(t^{2}+\zeta)^{(\alpha+\beta)/2}
-\beta\sin(\pi\alpha)t^{\alpha+\beta}
\right] e^{-t} dt,
\end{multline}
%%%%%%%%%%%%%%%%%
with a similar formula for $\partial\Psi_1(\alpha,\beta,\zeta)/\partial\zeta$. For each fixed $\zeta>0$, differentiation under the integral signs is justified by dominated convergence: near $t=0$ the factors involving $t^2+\zeta$ are bounded and the remaining powers of $t$ are integrable for $\alpha,\beta\in(0,1)$, while as $t\to\infty$ the resulting integrands are bounded by a power of $t$ times $e^{-t}$ and are therefore integrable. The conclusions at $\zeta=0$ then follow by continuity.

Now, if $0<\beta\le\alpha<1$, then
%%%%%%%%%%%%%%%%%
\begin{equation}
\label{eq148}
\alpha\sin(\pi\beta)
\ge
\beta\sin(\pi\alpha),
\end{equation}
%%%%%%%%%%%%%%%%%
and
%%%%%%%%%%%%%%%%%
\begin{equation}
\label{eq149}
(t^2+\zeta)^{(\alpha+\beta)/2}
\ge t^{\alpha+\beta}.
\end{equation}
%%%%%%%%%%%%%%%%%
Hence \cref{eq148,eq149}, applied to \cref{eq147} and the corresponding formula for \linebreak $\partial\Psi_1(\alpha,\beta,\zeta)/\partial\zeta$, give
$\partial\Psi_0(\alpha,\beta,\zeta)/\partial\zeta\ge0$ and
$\partial\Psi_1(\alpha,\beta,\zeta)/\partial\zeta\ge0$.
This proves the assertions of the lemma when $0<\beta\le\alpha<1$.

It remains only to prove that $\Psi_0(\alpha,\beta,\zeta)$ is increasing when $0<\alpha<\beta<1$. 
On making the change of variable $t=\sqrt{\zeta}\,x$ in \cref{eq147}, we obtain
%%%%%%%%%%%%%%%%%
\begin{equation}
\label{eq150}
\frac{\partial \Psi_{0}(\alpha,\beta,\zeta)}
{\partial \zeta}
=
\frac12\,\zeta^{(\alpha-\beta-1)/2}
\mathcal G(\alpha,\beta,\sqrt{\zeta}),
\end{equation}
%%%%%%%%%%%%%%%%%
where
%%%%%%%%%%%%%%%%%
\begin{equation}
\label{eq151}
\mathcal G(\alpha,\beta,s)
=
\int_{0}^{\infty}q(\alpha,\beta,x)e^{-sx}\,dx
\quad (s\ge0),
\end{equation}
%%%%%%%%%%%%%%%%%
%%%%%%%%%%%%%%%%%
\begin{equation}
\label{eq152}
q(\alpha,\beta,x)
=
x^{-\beta}(1+x^{2})^{-1-\beta/2}
\left\{
\kappa(\alpha,\beta)(1+x^{2})^{(\alpha+\beta)/2}
-\chi(\alpha,\beta)x^{\alpha+\beta}
\right\},
\end{equation}
%%%%%%%%%%%%%%%%%
and
%%%%%%%%%%%%%%%%%
\begin{equation}
\label{eq153}
\kappa(\alpha,\beta)=\alpha\sin(\pi\beta),
\quad
\chi(\alpha,\beta)=\beta\sin(\pi\alpha).
\end{equation}
%%%%%%%%%%%%%%%%%
Since $x\mapsto\sin(\pi x)/x$ is strictly decreasing on $(0,1)$,
we have
%%%%%%%%%%%%%%%%%
\begin{equation}
\label{eq154}
\kappa(\alpha,\beta)<\chi(\alpha,\beta).
\end{equation}
%%%%%%%%%%%%%%%%%
Now
%%%%%%%%%%%%%%%%%
\begin{equation}
\label{eq155}
x^{-\alpha-\beta}(1+x^2)^{(\alpha+\beta)/2}
=
(1+x^{-2})^{(\alpha+\beta)/2}
\end{equation}
%%%%%%%%%%%%%%%%%
is strictly decreasing for $x>0$. Hence from \cref{eq152,eq154} there exists a unique
$x_0>0$ such that
%%%%%%%%%%%%%%%%%
\begin{equation}
\label{eq156}
q(\alpha,\beta,x)>0\quad(0<x<x_0),
\quad
q(\alpha,\beta,x)<0\quad(x>x_0).
\end{equation}
%%%%%%%%%%%%%%%%%
Writing
%%%%%%%%%%%%%%%%%
\begin{equation}
\label{eq157}
P(\alpha,\beta,s)=\int_0^{x_0}q(\alpha,\beta,x)e^{-sx}\,dx,
\quad
N(\alpha,\beta,s)=-\int_{x_0}^{\infty}q(\alpha,\beta,x)e^{-sx}\,dx,
\end{equation}
%%%%%%%%%%%%%%%%%
we have
%%%%%%%%%%%%%%%%%
\begin{equation}
\label{eq158}
\mathcal G(\alpha,\beta,s)=P(\alpha,\beta,s)-N(\alpha,\beta,s)
=P(\alpha,\beta,s)\{1-T(\alpha,\beta,s)\},
\end{equation}
%%%%%%%%%%%%%%%%%
where
%%%%%%%%%%%%%%%%%
\begin{equation}
\label{eq159}
T(\alpha,\beta,s)=\frac{N(\alpha,\beta,s)}{P(\alpha,\beta,s)}.
\end{equation}
%%%%%%%%%%%%%%%%%
A direct calculation gives
%%%%%%%%%%%%%%%%%
\begin{multline}
\label{eq160}
T'(s)
=
\frac{N'(\alpha,\beta,s)P(\alpha,\beta,s)-N(\alpha,\beta,s)P'(\alpha,\beta,s)}{P(\alpha,\beta,s)^2}
\\
=
-\frac{1}{P(\alpha,\beta,s)^2}
\int_{x_0}^{\infty}\int_0^{x_0}
(x-y)\{-q(\alpha,\beta,x)\}q(\alpha,\beta,y)e^{-s(x+y)}
\,dy\,dx<0,
\end{multline}
%%%%%%%%%%%%%%%%%
and thus $T(\alpha,\beta,s)$ is strictly decreasing. Next, direct evaluation of \cref{eq151,eq152} gives
%%%%%%%%%%%%%%%%%
\begin{equation}
\label{eq161}
\mathcal G(\alpha,\beta,0)
=
\frac12\Gamma\left(\tfrac12+\tfrac12\beta
-\tfrac12\alpha\right)
\left\{\kappa(\alpha,\beta)\,
\frac{\Gamma\left(\tfrac12-\tfrac12\beta\right)}
     {\Gamma\left(1-\tfrac12\alpha\right)}
-\chi(\alpha,\beta)\,
\frac{\Gamma\left(\tfrac12+\tfrac12\alpha\right)}
     {\Gamma\left(1+\tfrac12\beta\right)}
\right\}.
\end{equation}
%%%%%%%%%%%%%%%%%
From basic trigonometric, and gamma reflection and duplication formulas \cite[Eqs. 4.21.27, 5.5.3 and 5.5.5]{NIST:DLMF} it can be expressed as
%%%%%%%%%%%%%%%%%
\begin{equation}
\label{eq162}
\mathcal G(\alpha,\beta,0)
=
\frac{\pi\beta\sin(\tfrac12\pi\alpha)
\Gamma(\tfrac12+\tfrac12\beta-\tfrac12\alpha)}
{\Gamma(\tfrac12-\tfrac12\alpha)
\Gamma(1+\tfrac12\beta)}
\left\{V(\alpha,\beta)-1\right\},
\end{equation}
%%%%%%%%%%%%%%%%%
where
%%%%%%%%%%%%%%%%%
\begin{equation}
\label{eq163}
V(\alpha,\beta)
=\frac{\alpha\, \Gamma(\tfrac12-\tfrac12\alpha)}
{\sin(\tfrac12\pi\alpha)\Gamma(1-\tfrac12\alpha)}
\frac{\sin(\tfrac12\pi\beta)
\Gamma(1+\tfrac12\beta)}
{\beta\,\Gamma(\tfrac12+\tfrac12\beta)}.
\end{equation}
%%%%%%%%%%%%%%%%%
Let
%%%%%%%%%%%%%%%%%
\begin{equation}
\label{eq164}
B(x)=\int_0^{\pi/2}\cos^x(t)\,dt
\quad (x>-1),
\end{equation}
%%%%%%%%%%%%%%%%%
then the beta integral \cite[Eqs. 5.4.6, 5.12.1 and 5.12.2]{NIST:DLMF} gives
%%%%%%%%%%%%%%%%%
\begin{equation}
\label{eq165}
\frac{\Gamma(\tfrac12+\tfrac12 x)}
{\Gamma(1+\tfrac12 x)}
=
\frac{2}{\sqrt{\pi}}B(x).
\end{equation}
%%%%%%%%%%%%%%%%%
For $0<x<1$, using strict concavity of $\ln\{\cos(t)\}$ on $[0,\tfrac12\pi)$, together with $\ln\{\cos(0)\} \allowbreak =0$, one can verify that
%%%%%%%%%%%%%%%%%
\begin{equation}
\label{eq166}
\cos(xt)>\cos^x(t)
\quad (0<t<\tfrac12\pi),
\end{equation}
%%%%%%%%%%%%%%%%%
and hence
%%%%%%%%%%%%%%%%%
\begin{equation}
\label{eq167}
B(x)<\int_0^{\pi/2}\cos(xt)\,dt
=\frac{\sin(\tfrac12\pi x)}{x}.
\end{equation}
%%%%%%%%%%%%%%%%%
Also, for $0<x<1$ and $0<t<\tfrac12\pi$, we have $\cos^{-x}(t)>1>\cos(xt)$, and therefore
%%%%%%%%%%%%%%%%%
\begin{equation}
\label{eq168}
B(-x)
>\int_0^{\pi/2}\cos(xt)\,dt
=\frac{\sin(\tfrac12\pi x)}{x}.
\end{equation}
%%%%%%%%%%%%%%%%%
It now follows from \cref{eq163,eq165,eq167,eq168} that $V(\alpha,\beta)>1$ for $0<\alpha,\beta<1$, and hence from \cref{eq162}
%%%%%%%%%%%%%%%%%
\begin{equation}
\label{eq169}
\mathcal G(\alpha,\beta,0)>0.
\end{equation}
%%%%%%%%%%%%%%%%%
Since $T(\alpha,\beta,s)$ is strictly decreasing, \cref{eq158,eq169} show that $\mathcal G(\alpha,\beta,s)>0$ for every $s\ge0$. Hence by \cref{eq150}, $\partial\Psi_0(\alpha,\beta,\zeta)/\partial\zeta>0$ for $\zeta>0$, proving the remaining assertion and completing the proof.
\end{proof}

%%%%%%%%%%%%%%%%%
\begin{lemma}
\label{lem:Psi-bounds}
For $\alpha,\beta\in(0,1)$ and $\zeta\in[0,81]$, the function \cref{eq63} satisfies
%%%%%%%%%%%%%%%%%
\begin{equation}
\label{eq170}
\mathcal H_5(\alpha,\beta,\zeta)\ge \mathcal H_5^*(\alpha,\beta),
\end{equation}
%%%%%%%%%%%%%%%%%
where
%%%%%%%%%%%%%%%%%
\begin{equation}
\label{eq171}
\mathcal H_5^*(\alpha,\beta)
=
\frac{H_5(\alpha,\beta,81)
-0.0254\,\Gamma(2-\beta)\sigma(\alpha,\beta)4^{-2+\beta}}
{\alpha\beta(\alpha+\beta)}
\quad (\alpha\ge\beta),
\end{equation}
%%%%%%%%%%%%%%%%%
%%%%%%%%%%%%%%%%%
\begin{multline}
\label{eq172}
\mathcal H_5^*(\alpha,\beta)
=
\frac{1}{\alpha\beta(\alpha+\beta)}
\Bigl[
C_5(\alpha,\beta)
-5^{-1-\alpha+\beta}\Psi_0(\alpha,\beta,81)
\\
+\frac12(\beta-\alpha)5^{-2-\alpha+\beta}
\left\{
\sin(\pi\beta)\Gamma(2+\alpha-\beta)
+\sin(\pi\alpha)M_{2,1}(\alpha,\beta,81)
\right\}
\\
-0.0254\cdot 4^{-2+\beta}\,\Gamma(2-\beta)\sigma(\alpha,\beta)
\Bigr]
\quad (\alpha<\beta),
\end{multline}
%%%%%%%%%%%%%%%%%
where $C_5(\alpha,\beta)$ is independent of $\zeta$, namely
%%%%%%%%%%%%%%%%%
\begin{multline}
\label{eq173}
C_5(\alpha,\beta)
=
\sin(\pi\beta)2^\alpha\Gamma(1-\beta)5^{-1+\beta}
\left\{
1+\frac{(\alpha-\beta)(1-\beta)}{10}
\right\}
\\
+\sin(\pi\alpha)2^{-\beta}\Gamma(1+\alpha)5^{-1-\alpha}
\left\{
1+\frac{(\alpha-\beta)(1+\alpha)}{10}
\right\}.
\end{multline}
%%%%%%%%%%%%%%%%%
Moreover, the function \cref{eq72} satisfies
%%%%%%%%%%%%%%%%%
\begin{equation}
\label{eq174}
\mathcal Q_5(\alpha,\beta,\zeta)\ge \mathcal Q_5^*(\alpha,\beta),
\end{equation}
%%%%%%%%%%%%%%%%%
where
%%%%%%%%%%%%%%%%%
\begin{multline}
\label{eq175}
\mathcal Q_5^*(\alpha,\beta)
=
\frac{1}{\alpha\beta(\alpha+\beta)}
\Bigl[
h_1(\alpha,\beta)5^\alpha
+h_2(\alpha,\beta)5^{-\beta}
\\
-(1-\alpha)\Psi_0(\alpha,\beta,81)
\Bigr]
\quad (\alpha\ge\beta),
\end{multline}
%%%%%%%%%%%%%%%%%
%%%%%%%%%%%%%%%%%
\begin{multline}
\label{eq176}
\mathcal Q_5^*(\alpha,\beta)
=
\frac{1}{\alpha\beta(\alpha+\beta)}
\left[
h_1(\alpha,\beta)5^\alpha
+h_2(\alpha,\beta)5^{-\beta}
-(1-\alpha)\Psi_0(\alpha,\beta,81)
\right.
\\
\left.
-\frac{1}{10}\alpha(\beta-\alpha)
\left\{
\sin(\pi\beta)M_{1,1}(\alpha,\beta,81)
+\sin(\pi\alpha)\Gamma(2+\alpha-\beta)
\right\}
\right]
\quad (\alpha<\beta).
\end{multline}
%%%%%%%%%%%%%%%%%
\end{lemma}

\begin{remark}
The endpoint cases away from $(\alpha,\beta)=(0,0)$ follow by continuous extension, with the behaviour at this corner being covered by \cref{eq186,eq187} below. By \cref{eq46}, the values at $\zeta=0$ used in deriving these bounds reduce to explicit gamma values. Thus numerical evaluation of the Meijer $G$-functions in \cref{eq171,eq172,eq175,eq176} is required only at the fixed endpoint $\zeta=81$.
\end{remark}

\begin{proof}
From \cref{lem:alpha<beta}, $\Psi_0(\alpha,\beta,\zeta)$ is increasing with respect to $\zeta$ for all $\alpha,\beta\in(0,1)$, while $\Psi_1(\alpha,\beta,\zeta)$ is increasing when $\alpha\ge\beta$. Thus, if $\alpha\ge\beta$, the monotonicity of $\Psi_0(\alpha,\beta,\zeta)$ and $\Psi_1(\alpha,\beta,\zeta)$ shows from \cref{eq56} that
%%%%%%%%%%%%%%%%%
\begin{equation}
\label{eq177}
H_5(\alpha,\beta,\zeta)\ge H_5(\alpha,\beta,81).
\end{equation}
%%%%%%%%%%%%%%%%%
This, together with \cref{eq63}, gives \cref{eq170,eq171}.

Now suppose that $\alpha<\beta$. We shall use the immediate consequences of \cref{eq41,eq43} that $M_{1,1}(\alpha,\beta,\zeta)$ is increasing and $M_{2,1}(\alpha,\beta,\zeta)$ is decreasing with respect to $\zeta$. Together with \cref{eq46}, this gives
%%%%%%%%%%%%%%%%%
\begin{equation}
\label{eq178}
\Psi_1(\alpha,\beta,\zeta)\ge
\sin(\pi\beta)\Gamma(2+\alpha-\beta)
+\sin(\pi\alpha)M_{2,1}(\alpha,\beta,81).
\end{equation}
%%%%%%%%%%%%%%%%%
Since also $\Psi_0(\alpha,\beta,\zeta)\le\Psi_0(\alpha,\beta,81)$, substitution in \cref{eq56,eq63} gives \cref{eq170,eq172}.

Next, for $\mathcal Q_5(\alpha,\beta,\zeta)$, from \cref{eq69,eq145} and the monotonicity of $\Psi_0(\alpha,\beta,\zeta)$,
%%%%%%%%%%%%%%%%%
\begin{equation}
\label{eq179}
h_3(\alpha,\beta,\zeta)\ge
-(1-\alpha)\Psi_0(\alpha,\beta,81).
\end{equation}
%%%%%%%%%%%%%%%%%
If $\alpha<\beta$, the monotonicity of $M_{1,1}(\alpha,\beta,\zeta)$ and $M_{2,1}(\alpha,\beta,\zeta)$, together with \cref{eq46,eq70}, gives
%%%%%%%%%%%%%%%%%
\begin{equation}
\label{eq180}
h_4(\alpha,\beta,\zeta)\le
\alpha(\beta-\alpha)
\left\{
\sin(\pi\beta)M_{1,1}(\alpha,\beta,81)
+\sin(\pi\alpha)\Gamma(2+\alpha-\beta)
\right\}.
\end{equation}
%%%%%%%%%%%%%%%%%

Suppose instead that $\alpha\ge\beta$. By \cref{lem:alpha<beta,eq46},
%%%%%%%%%%%%%%%%%
\begin{equation}
\label{eq181}
\Psi_1(\alpha,\beta,\zeta)\ge
\{\sin(\pi\alpha)+\sin(\pi\beta)\}\Gamma(2+\alpha-\beta).
\end{equation}
%%%%%%%%%%%%%%%%%
Moreover, logarithmic convexity of the gamma function gives
%%%%%%%%%%%%%%%%%
\begin{equation}
\label{eq182}
\Gamma(x+\beta)
\le \Gamma(x)^{1-\beta}\Gamma(x+1)^\beta
=x^\beta\Gamma(x)
\end{equation}
%%%%%%%%%%%%%%%%%
for $x>0$ and $0\le\beta\le1$.
Taking $x=2+\alpha-\beta\in[2,3]$, we obtain
%%%%%%%%%%%%%%%%%
\begin{equation}
\label{eq183}
\Gamma(2+\alpha)10^{-\beta}
\le \Gamma(2+\alpha-\beta).
\end{equation}
%%%%%%%%%%%%%%%%%
Combining \cref{eq181,eq183} with the sine inequality \cref{eq148} gives
%%%%%%%%%%%%%%%%%
\begin{multline}
\label{eq184}
\sin(\pi\alpha)(\alpha+\beta)\Gamma(2+\alpha)10^{-\beta}
\\
\le
\alpha\{\sin(\pi\alpha)+\sin(\pi\beta)\}\Gamma(2+\alpha-\beta)
\le
\alpha\Psi_1(\alpha,\beta,\zeta).
\end{multline}
%%%%%%%%%%%%%%%%%
Multiplying by $\alpha-\beta\ge0$ and using \cref{eq70,eq146},
we conclude that $h_4(\alpha,\beta,\zeta)\le0$. Hence
%%%%%%%%%%%%%%%%%
\begin{equation}
\label{eq185}
[h_4(\alpha,\beta,\zeta)]_+=0
\qquad (\alpha\ge\beta).
\end{equation}
%%%%%%%%%%%%%%%%%
Substitution of \cref{eq179,eq180,eq185} in \cref{eq72}
gives \cref{eq174,eq175,eq176}.
\end{proof}

Although $\mathcal H_5^*(\alpha,\beta)$ and $\mathcal Q_5^*(\alpha,\beta)$ do not have unique limits as $(\alpha,\beta)\to(0^+,0^+)$, this corner does not affect their
numerical minima. Specifically, expanding \cref{eq171,eq172,eq175,eq176} jointly as $(\alpha,\beta)\to(0^+,0^+)$, with remainders that are uniform in $\alpha/(\alpha+\beta)\in(0,1)$, yields
%%%%%%%%%%%%%%%%%
\begin{equation}
\label{eq186}
\liminf_{(\alpha,\beta)\to(0^+,0^+)}
\mathcal H_5^*(\alpha,\beta)
=0.0588262521\cdots,
\end{equation}
%%%%%%%%%%%%%%%%%
and
%%%%%%%%%%%%%%%%%
\begin{equation}
\label{eq187}
\liminf_{(\alpha,\beta)\to(0^+,0^+)}
\mathcal Q_5^*(\alpha,\beta)
=9.6200462499\cdots.
\end{equation}
%%%%%%%%%%%%%%%%%
These values exceed the lower bounds in \cref{eq64,eq73}, respectively. Thus the corner $\alpha=\beta=0$ does not affect the minima used in these numerical verifications.

The numerical verification of \cref{eq64,eq73} was carried out by evaluating the two-dimensional lower bounds $\mathcal H_5^*(\alpha,\beta)$ and $\mathcal Q_5^*(\alpha,\beta)$ of \cref{lem:Psi-bounds} on a fine grid in $(\alpha,\beta)\in[0,1]^2$. The grid was refined near the edges of the square, where the normalising factor $\alpha\beta(\alpha+\beta)$ vanishes, and the corresponding boundary values away from $\alpha=\beta=0$ were evaluated by continuous extension. The smallest value of $\mathcal H_5^*(\alpha,\beta)$ on the grid occurred near the edge $\alpha=0$. We therefore took the continuous limit as $\alpha\to0$ and minimised the resulting one-variable function of $\beta$, obtaining the lower bound $0.0388137394\cdots$ in \cref{eq64}, at $\beta=0.33900146\cdots$. Similarly, the grid minimum of $\mathcal Q_5^*(\alpha,\beta)$ occurred near the corner $(\alpha,\beta)=(0,1)$. Taking the corresponding continuous boundary limit gave the value found in \cref{eq73}.

As an independent check, we introduced the transformed variables $\alpha=1/(1+\exp(-a))$ and $\beta=1/(1+\exp(-b))$, with $-30\le a,b\le30$, which allowed the search to approach exponentially close to both edges $\alpha,\beta=0,1$. We then divided the resulting square into bounded cells, and in each cell applied Maple's \texttt{Optimization[Minimize]} routine separately, starting at its centre. The calculations were performed using 60-digit arithmetic and gave lower bounds consistent with those obtained from the grid search.

The surface shown in \cref{fig:M5} illustrates the lower bound $\mathcal H_5^*(\alpha,\beta)$ used to obtain \cref{eq64}, and \cref{fig:Mp5} depicts $\mathcal Q_5^*(\alpha,\beta)$ used to obtain \cref{eq73}. Both are consistent with the infima found numerically, at $\alpha=0$, $\beta=0.33900146\cdots$ for $\mathcal H_5^*(\alpha,\beta)$, and at $(\alpha,\beta)=(0,1)$ for $\mathcal Q_5^*(\alpha,\beta)$.

\begin{figure}[hthp]
 \centering
 \includegraphics[
 width=0.9\textwidth,keepaspectratio]{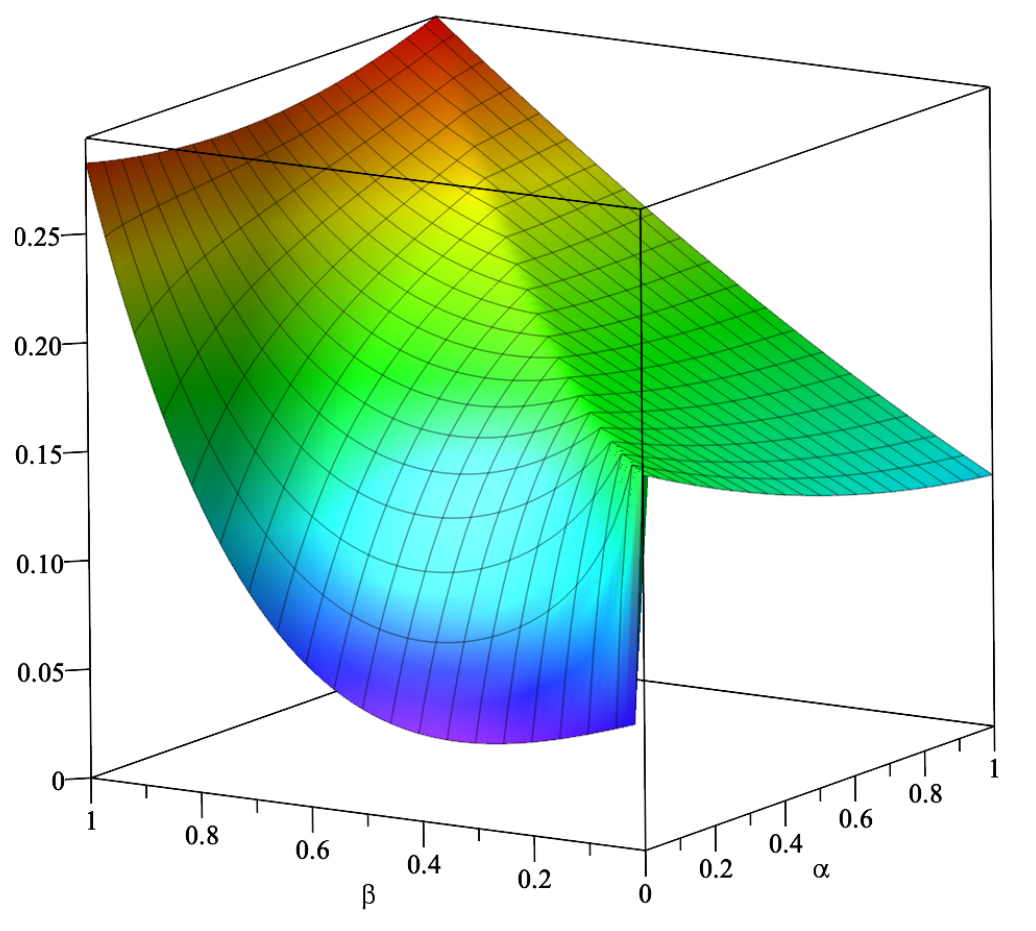}
 \caption{Graph of $\mathcal H_5^*(\alpha,\beta)$ for $\alpha,\beta \in [0,1]$.}
 \label{fig:M5}
\end{figure}

\begin{figure}[htbp]
 \centering
 \includegraphics[
 width=0.9\textwidth,keepaspectratio]{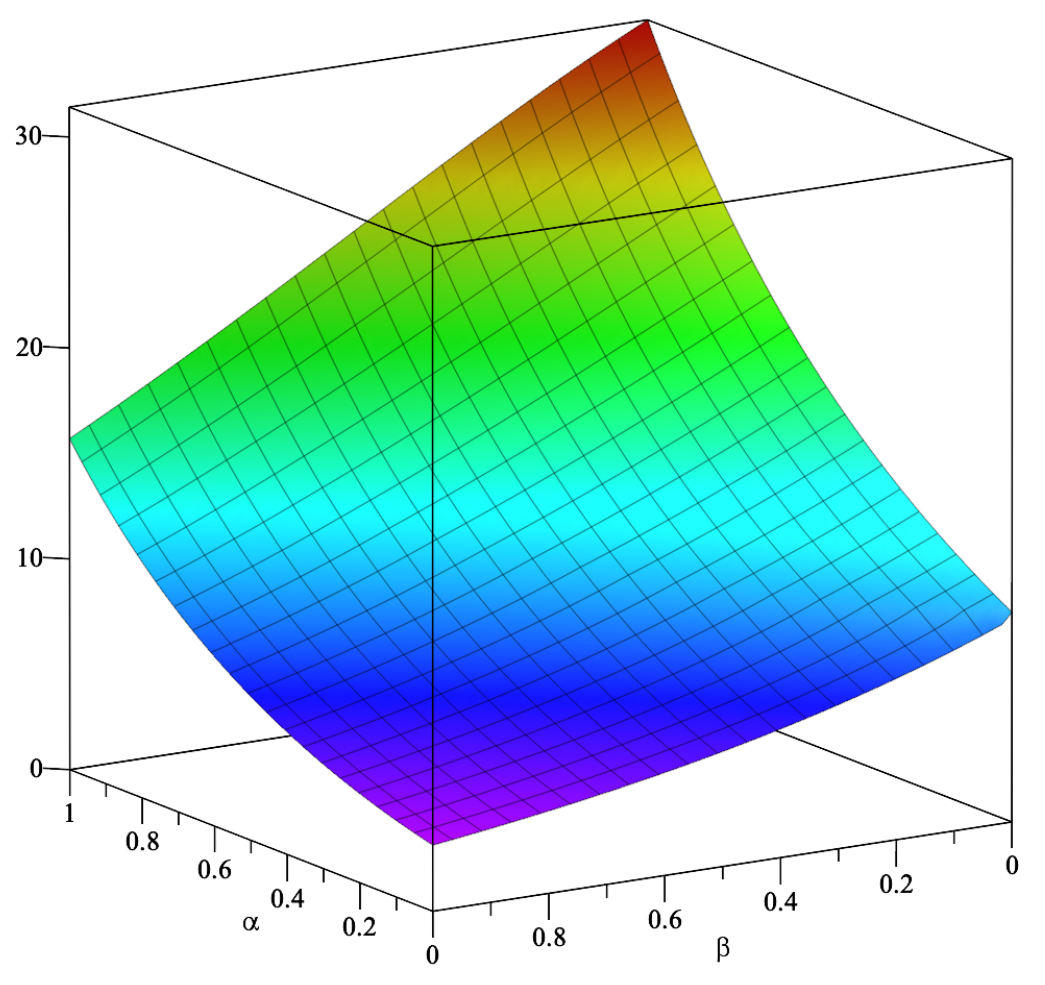}
 \caption{Graph of $\mathcal Q_5^*(\alpha,\beta)$ for $\alpha,\beta \in [0,1]$.}
 \label{fig:Mp5}
\end{figure}

The numerical verification of \cref{eq94} was carried out similarly on a grid in $(\alpha,\beta)\in[0,1]^2$, again with additional grid points near the boundary values $0$ and $1$, but in this case the additional variable $\phi\in[0,\phi_0]$ also had to be minimised. For each fixed grid point $(\alpha,\beta)$, we evaluated $\widetilde P(\alpha,\beta,\phi)$ at the endpoints $\phi=0$ and $\phi=\phi_0$, and used \texttt{fsolve} with sign-change bracketing to locate any interior stationary points for $\phi \in (0,\phi_0)$. The function was then evaluated at each of these points and the smallest value retained. Repeating this procedure over the $(\alpha,\beta)$ grid gave the minimum stated in \cref{eq94}, attained at $(\alpha,\beta,\phi)=(1,1,\phi_0)$.

For a separate verification, we used the same transformed variables for $\alpha$ and $\beta$ as above, together with $\phi=\phi_0\exp(-p)$, $0\le p\le30$. To remove the nonsmoothness arising from the positive-part terms in \cref{eq90} and the two definitions of $\varpi(\alpha,\beta,\phi)$ in \cref{eq86}, the transformed $(a,b)$-square was divided along $a=b$ and $b=-\ln 3$, the latter corresponding to $\beta=1/4$. The resulting triangular regions were mapped to rectangles by elementary parametrisations, and all six smooth regions were subdivided into bounded cells. Maple's \texttt{Optimization[Minimize]} routine was then applied separately in each cell, starting at its centre. With $60$-digit arithmetic and $8^3$ cells in each region, all $3072$ minimisations succeeded, with results consistent with the grid search.

For the computation leading to \cref{eq106}, we minimised $\mathsf L_0(\alpha,\beta,\rho,s)$ over $\alpha,\beta\in[0,1]$, $\rho\in[0,2\sin(\tfrac12\phi_0)]$, and $s\in[0,10]$. The parameter variables $(\alpha,\beta,\rho)$ were sampled on grids biased toward the boundary values. For each fixed triple $(\alpha,\beta,\rho)$, the minimisation in $s$ was carried out by checking the endpoint $s=10$ together with any stationary points in $(0,10)$ found from the corresponding derivative equation using \texttt{fsolve} with sign-change bracketing. The grid minimum occurred near the boundary $(\alpha,\beta,\rho)=(0,1,0)$. Taking the corresponding boundary limit and minimising the resulting one-variable function of $s$ gives the value and minimising point stated in \cref{eq106}. A separate four-dimensional check was then made using the same exponential parametrisations for $\alpha$ and $\beta$ as above, together with an exponential parametrisation of $\rho$ approaching the endpoints of its range and $s=10\exp(-q)$, $0\le q\le30$, and the resulting transformed domain was divided into bounded cells. Maple's \texttt{Optimization[Minimize]} routine was applied separately in each cell, starting at its centre, and gave results consistent with the grid calculation.

For the computation leading to \cref{eq138}, we proceeded similarly by evaluating $\widetilde{\mathsf L}_0(\alpha,\beta,\phi,s)$ on a grid in $(\alpha,\beta,\phi)$ over $[0,1]^2\times[0,\phi_0]$. For each fixed triple $(\alpha,\beta,\phi)$, the minimisation with respect to $s\in[0,10]$ was carried out by the same stationary-point procedure as above. The grid calculation showed that the infimum is approached at the corner $\alpha=\beta=0$, with $s$ and $\phi$ tending to zero together. From \cref{eq132,eq136}, setting $s=c\phi$ gives
%%%%%%%%%%%%%%%%%
\begin{equation}
\label{eq188}
\widetilde{\mathsf L}_{0}(0,0,\phi;c\phi)
\to \frac{1}{8\pi\sqrt{c}}\ln(1+c^2)
\ln\left(\frac{c^4}{1+c^2}\right)
\quad (\phi\to0,\ c>0).
\end{equation}
%%%%%%%%%%%%%%%%%
The minimum of this limiting expression on $c\in(0,\infty)$ is found by elementary calculus to be $-0.0377104213\cdots$, attained at $c=c_0=0.5281398736\cdots$, in agreement with both the grid search and \cref{eq138}. As a separate check, we used the same exponential parametrisations for $\alpha$, $\beta$, and $\phi$ as above, together with $s=10\exp(-q)$, $0\le q\le30$, and divided the resulting four-dimensional domain into bounded cells. Maple's \texttt{Optimization[Minimize]} routine was applied separately in each cell, starting at its centre, and the overall minimum was consistent with the grid and limiting calculations.

For the remaining tail $s\ge10$, we set $t=1/s$, so that $0\le t\le0.1$, and minimised with respect to $(\alpha,\beta,\phi,t)$ lying in a compact box. The same exponential rescalings for $\alpha$, $\beta$, and $\phi$ as above were used, while $t \in [0,0.1]$ did not require rescaling. Maple's \texttt{Optimization[Minimize]} was applied on the resulting four-dimensional domain, with the variables divided into bounded cells. For both $\mathsf L_0(\alpha,\beta,\rho,s)$ and $\widetilde{\mathsf L}_0(\alpha,\beta,\phi,s)$, the tail produced no smaller value than that obtained for $0\le s\le10$.

We conclude by considering the evaluation of the functions $M_{j,k}(\alpha,\beta,\zeta)$ using the Meijer $G$-function representations \cref{eq51,eq52}. The bounds in \cref{lem:Psi-bounds} require numerical Meijer $G$-function evaluations only at $\zeta=81$, since the values at $\zeta=0$ are given explicitly by \cref{eq46}. To test the overall robustness of Maple's built-in Meijer $G$-function routines over a wider range of $\zeta$, we checked their evaluations against independent MATLAB quadrature of the defining integrals \cref{eq40,eq41,eq42,eq43} on the grid $\zeta=1,\ldots,81$ and $\alpha,\beta\in\{2\times10^{-6},0.1,0.2,\ldots,0.9,1-2\times10^{-6}\}$. These checks are described below.

The defining integrals \cref{eq40,eq41,eq42,eq43} can all be written in the canonical form
%%%%%%%%%%%%%%%%%
\begin{equation}
\label{eq189}
I(a,b,\zeta)=\int_0^\infty t^{-b}(t^2+\zeta)^{a/2}e^{-t}\,dt.
\end{equation}
%%%%%%%%%%%%%%%%%
For the endpoint-sensitive cases, especially when $b$ is close to $1$, we stabilised the quadrature by integrating by parts. This gives
%%%%%%%%%%%%%%%%%
\begin{equation}
\label{eq190}
I(a,b,\zeta)
=
\frac{1}{1-b}\int_0^\infty t^{1-b}
\left\{(t^2+\zeta)^{a/2}
-a t(t^2+\zeta)^{a/2-1}\right\}e^{-t}\,dt
\quad (b<1).
\end{equation}
%%%%%%%%%%%%%%%%%
The cases \cref{eq40,eq41,eq42,eq43} were then obtained by taking $M_{1,0}(\alpha,\beta,\zeta)=I(\alpha,\beta,\zeta)$, $M_{1,1}(\alpha,\beta,\zeta)=I(\alpha,\beta-1,\zeta)$, $M_{2,0}(\alpha,\beta,\zeta)=I(-\beta,-\alpha,\zeta)$, and $M_{2,1}(\alpha,\beta,\zeta)=I(-\beta,-1-\alpha,\zeta)$.

The MATLAB quadrature values ($M^{(Q)}$, say), computed using double-precision adaptive quadrature, were compared with the Maple Meijer $G$ values ($M^{(G)}$, say) on the same grid. The relative error was computed as $|M^{(G)}-M^{(Q)}|/M^{(Q)}$. We found that the largest relative errors were approximately $2.68\times10^{-11}$, $6.22\times10^{-15}$, $1.34\times10^{-11}$, and $1.02\times10^{-15}$ for $M_{1,0}(\alpha,\beta,\zeta)$, $M_{1,1}(\alpha,\beta,\zeta)$, $M_{2,0}(\alpha,\beta,\zeta)$, and $M_{2,1}(\alpha,\beta,\zeta)$, respectively. The largest absolute discrepancies occurred near $\beta=1$, where the corresponding function values are large and the MATLAB double-precision quadrature is least stable. In the remaining cases, the agreement was close to the accuracy expected from double-precision quadrature.

For the most endpoint-sensitive cases in this grid, the observed discrepancies are dominated by the MATLAB double-precision quadrature rather than by Maple's Meijer $G$ evaluation. To check this, we recomputed representative worst-case values using high-precision Maple quadrature, stabilised by \cref{eq190}, and compared these with the Maple Meijer $G$ routines. For example, setting Digits equal to 100, for $\alpha=\beta=0.999998$ and $\zeta=81$, the Maple quadrature and Maple Meijer $G$ values agreed to essentially the full 100-digit precision requested.

The Maple source code used for the numerical computations described in this section and animated versions of \cref{fig:M5,fig:Mp5,fig:tildeP,fig:Meijer-m0,fig:tildem0} are available at \cite{Dunster:2026:BMM}.

\section*{Acknowledgement}
Financial support from Ministerio de Ciencia e Innovación project PID2024-159583NB-I00 (MICIU/ AEI / 10.13039/501100011033 / FEDER, UE) is acknowledged.

\section*{Conflict of interest}
The author declares no conflicts of interest.

\makeatletter
\interlinepenalty=10000

\bibliographystyle{siamplain}
\bibliography{biblio}
\end{document}